\UseRawInputEncoding
\documentclass[12pt]{article}
\usepackage{amscd}
\usepackage{latexsym,amsmath,amssymb,amsbsy, amsthm}
\usepackage{subcaption}

\usepackage[square]{natbib}

\usepackage{hyperref}
\usepackage{lscape,fancyhdr,fancybox}
\usepackage{graphicx,epsfig, xy}
\usepackage{color}
\usepackage[hmarginratio=1:2, vmarginratio =5:5,
textheight=22cm,bindingoffset=1.5cm, textwidth=14.6cm]{geometry}
\usepackage{tikz}

\newcommand{\comment}[1]{}

\numberwithin{equation}{section}

\newtheorem{remark}{Remark}[section]
\newtheorem{theorem}{Theorem}[section]

\newtheorem{lemma}{Lemma}[section]

\newtheorem{cor}{Corollary}[section]
\theoremstyle{definition}
\newtheorem{definition}{Definition}[section]

\newtheorem{result}{Result}[section]

\usepackage{amsfonts}

\usepackage{txfonts}

\DeclareMathOperator{\Tr}{Tr}

\newcommand{\beq}{\begin{eqnarray}}
\newcommand{\eeq}{\end{eqnarray}}
\newcommand{\ben}{\begin{eqnarray*}}
\newcommand{\een}{\end{eqnarray*}}

\allowdisplaybreaks

\begin{document}
\def\shorttitle{Wigner matrix}
 \def\shortauthors{A. Bose, K. Saha, A. Sen, P. Sen}
\title{\textbf{\Large \sc
\Large{Random matrices with independent entries: beyond non-crossing partitions} 
}}\small

 \author{
 \parbox[t]{0.20
\textwidth}{{\sc Arup Bose}
 \thanks{Statistics and Mathematics Unit, Indian Statistical Institute, 203 B.T. Road, Kolkata 700108, India. email: bosearu@gmail.com. 
Research  supported by J.C. Bose National Fellowship, Department of Science and Technology, Govt. of India.}}
\parbox[t]{0.25\textwidth}{{\sc Koushik Saha}
\thanks{Department of Mathematics, Indian Institute of Technology Bombay, Mumbai,  India. email: koushik.saha@iitb.ac.in.
Research  supported by MATRICS Grant of Science \& Engineering Research Board, Department of Science and Technology, Govt. of India.}}
\parbox[t]{0.25\textwidth}{{\sc Arusharka Sen}
\thanks{Department of Mathematics and Statistics,  Concordia University, 1455 Boulevard de Maisonneuve O, Montréal, QC H3G 1M8, CANADA. email: arusharka.sen@concordia.edu}}
 \parbox[t]{0.25\textwidth}{{\sc Priyanka Sen}
 \thanks{Statistics and Mathematics  Unit, Indian Statistical Institute, 203 B.T. Road, Kolkata 700108, INDIA. email: priyankasen0702@gmail.com}}
}

\date{\today}  
\maketitle
\begin{abstract}
The scaled standard Wigner matrix (symmetric with mean zero, variance one i.i.d. entries), and its limiting eigenvalue distribution, namely the semi-circular distribution, has attracted much attention. 
The $2k$th moment of the limit equals the number of non-crossing pair-partitions  of the set $\{1, 2, \ldots, 2k\}$. There are several extensions of this result in the literature. In this paper we consider a unifying extension which also yields additional results. 

Suppose $W_n$ is an $n\times n$ symmetric matrix  where the entries are independently distributed.
We show that under suitable assumptions on the entries, the limiting spectral distribution exists in probability or almost surely. The moments of the limit can be described through a set of partitions which in general is larger than the set of non-crossing pair-partitions.
This set gives rise to interesting enumerative combinatorial problems. 

Several existing limiting spectral distribution results follow from our results. These include results on the standard Wigner matrix, the adjacency matrix of a sparse homogeneous Erd\H{o}s-R\'{e}nyi graph, heavy tailed Wigner matrix, some banded Wigner matrices, and Wigner matrices with variance profile. Some new results on these models and their extensions also follow from our main results. 


\end{abstract}
\vskip 5pt
\noindent \textbf{Key words and phrases.} Empirical and limiting spectral distribution, Wigner matrix, semi-circular  distribution, non-crossing partition, special symmetric partition, cumulant and free cumulant, coloured rooted tree, graphon.  
\vskip 5pt
\noindent \textbf{AMS 2010 Subject Classifications.} Primary 60B20; Secondary  60B10.
\medskip

%
%
%
%
%
%
%
%
%
	%
	
\section{Introduction}\label{introduction}
Suppose $A_n$ is an $n\times n$ real \textit{symmetric} random matrix with (real) eigenvalues $\lambda_1,\lambda_2,\ldots,\lambda_n$. Its \textit{empirical spectral distribution} or measure (ESD) is the random probability measure:
\begin{align*}
\mu_{A_n} = \frac{1}{n} \sum_{i=1}^{n}  \delta_{\lambda_i},
\end{align*}
where $\delta_x$ is the Dirac measure at $x$. The expected empirical spectral distribution (EESD) is the corresponding expected empirical measure. 

The notions of convergence used for the ESD are (i) weak convergence of the EESD, and (ii) weak convergence of the ESD, the latter being either in probability or almost surely. The limits are identical when the latter limits are non-random. In any case, any of these limits will be referred to as the \textit{limiting spectral distribution} (LSD) of $\{A_n\}$. 

An $n\times n$ real symmetric random matrix 
\begin{align*}
W_n =  \begin{bmatrix}
x_{11} & x_{1 2} & x_{1 3} &\cdots & x_{1 (n-1)} & x_{1n}\\
x_{12} & x_{22} & x_{23} &\cdots & x_{2(n-1)} & x_{2n}\\ 
          \vdots  & \vdots & \vdots & \vdots & \vdots & \vdots\\
x_{1n} & x_{2n} & x_{3n} &\cdots & x_{(n-1)n} & x_{n n}          
\end{bmatrix},
\end{align*}
is called a Wigner matrix where it is often assumed that $\{x_{ij}\ ;1\leq i\leq j\leq n\}$ are independent and identically distributed (i.i.d.) random variables.
We mention five specific existing results for Wigner matrices under various assumptions on the entries. 
Some other results are quoted later in Section  \ref{discussion}. The first two are for fully i.i.d.~entries. For any set $A$, we use $|A|$ to denote the number of elements of $A$. For any probability measure $\mu$, $\beta_k(\mu)$ will denote the $k$th moment of $\mu$.
Let $NC(k)$ denote the set of non-crossing partitions on the set $\{1,2,\ldots, k\}$. Let $NC_2(2k)$ denotes the set of non-crossing pair-partitions of $\{1, 2, \ldots, 2k\}$.

See Theorem 2.1.3 in \citep{bose2018patterned} for a detailed proof and a short history of the precursors of the following result.
\begin{result}\label{result:semi-circle} (Standardized i.i.d. entries) Suppose that the entries $\{x_{ij}\ ;  i, j \geq 1, \ i  \leq j\}$ of $W_n$ are  i.i.d. 
with mean 0 and variance 1. Then, as $n \to \infty$, the almost sure LSD of $W_n/\sqrt{n}$ is the standard semi-circular distribution. This distribution, say $\mu_s$, has the following density
$$f(x)=\left\{
\begin{array}{cc}
\frac{1}{2\pi}\sqrt{4-x^2}&\mbox{if}\ x\in[-2,\ 2],\\
0& \mbox{otherwise}.
\end{array}
\right.$$
\end{result}
It is well-known that 
$$\beta_k(\mu_{s})= \left\{
\begin{array}{cc}
0 &\mbox{if}\ k \ \text{is odd},\\
\frac{2}{(k+1)} {k \choose {\frac{k}{2}}}= |NC_2(k)| & \mbox{if}\ k \ \mbox{is even}.
\end{array}
\right.$$

The next result deals with the case where the entries are heavy tailed. 
\begin{result}\label{Heavy-tailed entries, BenArous}(Theorem 1.1, \citep{arous2008spectrum})
Suppose the entries $\{x_{ij}\ ;i \leq j\}$ of $W_n$ are i.i.d. and satisfy $P\{|x_{ij}|> u\}= u^{-\alpha}L(u)$ as $u \to \infty$  where $L(\cdot)$ is  slowly varying and $\alpha \in (0,2)$. Also $\displaystyle \lim_{u \rightarrow \infty} \frac{\mathbb{P}[x_{ij}>u]}{\mathbb{P}[|x_{ij}|>u]}=\theta \in [0,1]$. 
Let $a_n=\inf\{u: P[|x_{ij}|> u]\leq 1/n\}$.
Then the ESD  of $\frac{W_{n}}{a_{n}}$ converges to a probability measure $\mu_{\alpha}$ in probability.   
\end{result}
A natural question is what happens if we retain the symmetry of the matrix but do not have i.i.d. entries and/or the distribution of the entries change as $n$ changes.  \textit{Hence we consider matrices $W_n=(x_{ij,n})$ where  $x_{ij,n}=x_{ji,n}$ for all $i, j$. For simplicity we drop the suffix $n$}. We quote three results  this scenario. 
The first result deals with a sparse case.
\begin{result}\label{result:sparse} 
(\citep{bauer2001random})
Suppose that for each fixed $n$, the entries of $W_n$ are i.i.d. Bernoulli $p_n$ such that $np_n\to \lambda$. Then the  EESD of $W_n$ converges weakly to a symmetric probability measure $\mu_{bg}$ say. Its even moments are given by 
\begin{align}\label{Bauer-exprss}
\beta_{2k}(\mu_{bg})=\sum_{b=1}^k I_{k,b}\lambda^b,
\end{align} 
where $I_{k,b}$ is the number of normalised $2k$-plets associated to certain trees with $b$ edges. 
\end{result}
A nicer description of $\mu_{bg}$ does not seem to be available, but it is easy to check that $\mu_{bg}$ is not a semi-circular distribution. Later, we shall present an alternative description of the above moments.

Now we quote a result where the primary motivation was to consider  entries of $W_n$ that are i.i.d. with distribution $G_n$ which may be all  light-tailed but as $n \to \infty$, $G_n$ may converge to a heavy-tailed distribution. 
\begin{result}\label{result:Zakharevich} (Theorem 1,  \citep{zakharevich2006generalization}) Suppose $\{G_k\}$ is a sequence of probability distributions each of which has mean zero and all moments finite. Let $\mu_n(k)$ be the $k$th moment of $G_n$. 
Suppose that for each $n$, the entries of $W_n$ are i.i.d. $G_n$. Let $$\displaystyle{\lim_{n\to\infty} \frac{\mu_n(k)}{n^{k/2-1}\mu_n(2)^{k/2}}}= g_k,  \ \ \text{say,  exists for all} \ \ k\geq 1.$$ Then the ESD of $\frac{W_n}{\sqrt{n\mu_n(2)}}$ converges in probability to a distribution $\mu_{zak}$ that depends only on the sequence $\{g_{2k}\}$.
\end{result}
Note that $g_2=1$. It is known that $\mu_{zak}=\mu_{s}$ if and only if $g_{2k}=0$ for all $k \geq 3$. This condition holds if $\{G_k\}$ are identical.

The moments of $\mu_{zak}$ in the above theorem are determined with the help of certain \textit{trees}. The construction of the trees in Result \ref{result:sparse} and Result \ref{result:Zakharevich} are different. We shall discuss these later in Sections \ref{sparse} and \ref{triangular iid}.   


The next result focuses on entries which have unequal variances. Suppose $\{\sigma_{ij}^2, i,j  \geq 2\}$ is an array of non-negative real numbers with $\sigma_{ij}^2=\sigma_{ji}^2$ for every $i,j$. For every $n \geq 1$, Let $I_1= [0,\frac{1}{n}]$ and $I_i= (\frac{i-1}{n}, \frac{i}{n}]$,   $i \geq 2$ be a partition of $[0, 1]$ and let 
\begin{equation*}H_n(x,y)= \sigma_{ij}^2 \ \ \text{when}\ \  (x,y) \in I_i\times I_j.
\end{equation*} This defines  a sequence of functions called \textit{graphons} $H_n$ on $[0, \ 1]^2$. 

Consider all finite multigraphs $G= (V, E)$ without loops with vertex set   $V=\{1,\ldots , n\}$ and edge set $E$. Let
$t(G, H_n)$ denote the \textit{homomorphism density}
\begin{equation}\label{zhu-density}
t(G, H_n)=\int_{[0,1]^{|V|}} \prod_{(i, j) \in E}H_n(x_i,x_j) \prod_{i \in V}dx_i.
\end{equation}
\begin{result}\label{result:Zhu}(Theorem 3.2, \citep{zhu2020graphon})
Suppose the entries $\{a_{ij}=a_{ji}: 1 \leq i \leq j \leq n\}$ of $W_n$ have mean zero. 
Let $\mathbb{E}[a_{ij}^2]=\sigma_{ij}^2$. Assume that 
\begin{equation}\label{boundedvar}\sup_{ij} \sigma_{ij}^2 \leq B < \infty,
\end{equation} 
 \begin{equation}\label{treecondition}\lim t(T, H_n) \ \text{exists for every finite tree} \ T,\ \ \text{and}
 \end{equation}
  \begin{equation}\label{lindeberg}
 \displaystyle \lim_{n \to \infty} \frac{1}{n^2} \sum_{1 \leq i, j \leq n} \mathbb{E}\big[a_{ij}^2\boldsymbol {1}_{[|a_{ij}|\geq \eta \sqrt{n}]}\big]=0 \ \ \ \  \text{ for any constant } \eta>0
\ \text{(Lindeberg's condition)}. 
\end{equation} 
Then under the conditions (\ref{boundedvar}), (\ref{treecondition}), and (\ref{lindeberg}), the LSD of $W_n/\sqrt{n}$, say $\mu_{zhu}$  exists almost surely. Its odd moments are $0$ and for every $k$,  
 \begin{equation}\label{zhulimit}\beta_{2k}(\mu_{zhu})=\sum_{T}\lim t(T, H_n)
 \end{equation}
 where the  sum is over all rooted trees $T$ each with $k+1$ vertices. 
\end{result}
%
One goal of this paper is to bring the above results under one umbrella. We state a general result in Theorem \ref{thm:maingk} 
along with an important Corollary \ref{thm:main}.
Results \ref{result:semi-circle}--\ref{result:Zhu} quoted above, essentially follow from these two results, and so do some more existing results which we mention later in  Section \ref{discussion}.
In addition, using the ideas from the proofs of Theorem \ref{thm:maingk} and Corollary \ref{thm:main}, in Results \ref{imhomoge}--\ref{res:block} we also prove some new results and/or improved versions of existing results for the adjacency matrices of some non-homogeneous Erd\H{o}s-R\'{e}nyi graphs, for matrices with variance profile, and for certain banded matrices and block random matrices. 
A detailed discussion on the existing and new results are given later in Section  \ref{discussion}. 
We provide a brief glimpse here.

Recall that $\mu_s$ is linked to $NC_{2}(2k)$.  
The limit distributions $\mu_{\alpha}, \mu_{bg}, \mu_{zak}$ and $\mu_{zhu}$ in Results \ref{Heavy-tailed entries, BenArous}--\ref{result:Zhu} respectively, cannot be described via $NC_{2}(2k)$. In Theorem \ref{thm:maingk}  we identify a larger class of partitions, called \textit{special symmetric partitions}, and denoted by $SS(2k)$, of $\{1, 2, \ldots , 2k\}$, which are linked to these limits. The class $SS(2k)$ includes $NC_{2}(2k)$ as well as many other crossing and non-crossing partitions. The contribution to the moments of the LSD in general vary across the partitions and depend on specific moment properties of the entries. 
The moments in Result \ref{result:Zakharevich} are described in \citep{zakharevich2006generalization} via certain trees. We connect that result with ours by showing that the set $SS(2k)$ is in one-to-one correspondence with those trees.   
Similarly, the count $I_{k,b}$ in (\ref{Bauer-exprss}) is the same as the count of a specific subclass $SS_{b}(2k)$ of $SS(2k)$. This provides an alternate partition description of the limiting moments of $\mu_{bg}$. 
Result \ref{result:Zhu} depends on only the first two moments of the entries. The limit moments of $\mu_{zhu}$ are described via certain rooted trees. Theorem \ref{thm:maingk} helps to generalise this result and thereby allowing contribution from the higher order moments of the entries. The contribution to the limit moments in the more general case now come from coloured rooted trees, which are also in one to one correspondence with $SS(2k)$.
The set $SS(2k)$ also appears in the limiting moments of the models described in Results \ref{res:graphons}--\ref{res:block} in Section \ref{discussion}. 
   
In  Section \ref{mainresults} we state our main result and its corollaries. In Section \ref{Circuits and Words} we provide some basic notions, specifically connected to partitions,  that we shall need in the proofs. 
In Section \ref{proof of the theorem}, we prove Theorem \ref{thm:maingk} and Corollary \ref{thm:main}.  In Section \ref{discussion} we connect our result with Results \ref{result:semi-circle}--\ref{result:Zhu} and  some other existing results. 
We also state and prove some new results and/or extend some existing results for specific models. 


\section{Main results}\label{mainresults}
Let $[k]:=\{1,2,\ldots, k\}$ and  $\mathcal{P}(k)$ denote the set of all partitions of $[k]$. Let $\mathcal{P}_{2}(2k)$ be the set of pair-partitions of $[2k]$.   
Suppose $\{c_k, \ k \geq 1\}$ is any sequence of numbers. Its multiplicative extension is defined on $\mathcal{P}_k$, $k \geq 1$ as follows. For any $\sigma\in \mathcal{P}_k$,  define 
\begin{equation*}c_{\sigma}=\prod_{V \ \ \text{is a block of}\ \ \sigma} c_{|V|}.
\end{equation*}
\begin{definition}(\textit{Special Symmetric Partition})\label{ss(2k)}
Let $\sigma \in \mathcal{P}(k)$ and let $V_1, V_2, \ldots $ be the blocks of $\sigma$, arranged in ascending order of their smallest elements.  This partition is said to be \textit{special symmetric} if 
\begin{enumerate}
\item[(i)] The last block is a union of sets of even size 
of consecutive integers.
\item[(ii)] Between any two successive elements of any block there are even number of elements from any other  block, and they each appear equal number of times in an odd and an even position.
\end{enumerate}
We denote the set of all special symmetric partitions of $[k]$ by $SS(k)$. 
\end{definition}
Clearly $SS(k)=\emptyset$ when $k$ is odd. Let $SS_b(2k)\subset SS(2k)$ where each partition has exactly $b$ blocks.  Clearly $b \leq k$ always. The one-block partition $\{1, 2, \ldots , 2k\}\in SS(2k)$. We shall see later that each block of $SS(2k)$ has even size. It is easy to check that 
$SS(2k)\cap \mathcal{P}_{2}(2k)=NC_2(2k) \subset SS(2k)$. There are $\pi \in SS(2k)$ that are either crossing or not paired. For example the partition  $\pi=\{\{1, 2, 5, 6\}, \{3, 4, 7, 8\}\}\in SS(8)$ but is crossing. 

Now we introduce a set of assumptions on the entries $\{x_{ij,n}\}$ of $W_n$. We drop the suffix $n$ for convenience wherever there is no confusion. For any real-valued function $g$ on $[0, \ 1]$, let $\|g\|:=\sup_{0 \leq x\leq 1} |g(x)|$ denote the sup norm. \\

\noindent \textbf{Assumption A}. 
Suppose $\{g_{2k,n}\}$ is a sequence of bounded Riemann integrable functions on $[0,1] ^2$. There exists a sequence $\{t_n\}$ with $t_n\in [0,\infty]$ such that 
\begin{enumerate}
\item [(i)] For each $k \in \mathbb{N}$,
\begin{align}
 & n \ \mathbb{E}\left[x_{ij}^{2k}\boldsymbol {1}_{\{|x_{ij}|\leq t_n\}}\right]= g_{2k,n}\Big(\frac{i}{n},\frac{j}{n}\Big) \ \ \ \text{for } \ 1\leq i \leq j \leq n, \label{gkeven}\\
& \displaystyle \lim_{n \rightarrow \infty} \ n^{\alpha} \underset{1\leq i \leq j \leq n}{\sup} \ \mathbb{E}\left[x_{ij}^{2k-1}\boldsymbol {1}_{\{|x_{ij}|\leq t_n\}}\right] = 0  \ \ \text{for any } \alpha<1 . \label{gkodd}
\end{align}
\item [(ii)] The functions $\{g_{2k,n}(\cdot), n \geq 1\}$ converges uniformly to $g_{2k}(\cdot)$ for each $k \geq 1$. 
\item[(iii)] Let  $M_{2k}=\|g_{2k}\|$, $M_{2k-1}=0$ for all $k \geq 1$. Then $\alpha_{2k}=  \sum_{\sigma \in \mathcal{P}(2k)} M_{\sigma}$  satisfies \textit{Carleman's condition}, 
$$ \displaystyle \sum_{k=1}^{\infty} \alpha_{2k}^{-\frac{1}{2k}}= \infty.$$
\end{enumerate}

\begin{theorem}\label{thm:maingk}
Let $W_n=(x_{ij,n})_{1\leq i\leq j \leq n}$ be the $n \times n$ symmetric matrix where $\{x_{ij,n}\ ;1 \leq i \leq j\leq n\}$ are independent and  satisfy Assumption A. Let $y_{ij}= x_{ij}\boldsymbol {1}_{\{|x_{ij}|\leq t_n\}}$ and $Z_n= (y_{ij,n})_{1\leq i\leq j \leq n}$. Then
\begin{enumerate}
\item[(a)] the ESD of $Z_n$ converges weakly almost surely to $\mu^\prime$ say, whose moment sequence is determined by the functions $g_{2k}, \ k \geq 1$.
\item[(b)] the ESD of $W_n$ converges weakly to $\mu^{\prime}$ almost surely (respectively in probability) if 
\begin{align}\label{truncation}
\frac{1}{n} \sum_{i,j} \ x_{ij}^2\boldsymbol {1}_{\{|x_{ij}| > t_n\}} \rightarrow 0, \ \mbox{almost surely (respectively in probability)},
\end{align}

 \end{enumerate} The law $\mu^{\prime}$ is semi-circular if and only if $g_{2k}=0$ for all $k>1$ and $\int_{0}^1 g_2(x,y)\ dy$ is constant.
\end{theorem}
\begin{remark}
 The condition that the sequence of functions $g_{2k,n}$ converges to $g_{2k}$ uniformly for all $k\geq 1$, can be weakened slightly. For details see Result \ref{res:graphons} in Section \ref{graphons}.
If all moments of $\{x_{ij}\}$ are finite, then choosing $t_n=\infty$, we have $Z_n=W_n$ and then the theorem concludes that the ESD of $W_n$ converges almost surely to $\mu^{\prime}$. 
\end{remark}
The following corollary is a special case of Theorem \ref{thm:maingk}.
\begin{cor}\label{thm:main}
Suppose the entries $\{x_{ij}=x_{ji}\ ;1 \leq i \leq j\leq n\}$ are i.i.d. for each $n$ and there is a sequence $(t_n)_{n \geq 1}$ such that the following two conditions hold:  
\begin{enumerate}
\item[(i)] For each $k \in \mathbb{N}$,
	 \begin{eqnarray}
\lim_{n \rightarrow \infty} n \ \mathbb{E}\left[x_{ij}^{2k}\boldsymbol {1}_{\{|x_{ij}|\leq t_n\}}\right]&=:& C_{2k} < \infty, \label{ckeven}\\
\lim_{n \rightarrow \infty} n^{\alpha} \  \mathbb{E}\left[x_{ij}^{2k-1}\boldsymbol {1}_{\{|x_{ij}|\leq t_n\}}\right] &=& 0 \ \ \text{for any}\ \  \alpha < 1.\label{ckodd}
\end{eqnarray}
\item[(ii)] The sequence $\{0, C_2, 0,C_4, 0,\ldots\}$ is  the cumulant sequence  of a probability distribution $G$ whose moment sequence $\{\beta_k\}$  satisfies Carleman's condition:
$$\displaystyle \sum_{k=1}^{\infty} \beta_{2k}^{-\frac{1}{2k}}= \infty.$$
\end{enumerate}
 Let $y_{ij}= x_{ij}\boldsymbol {1}_{[|x_{ij}|\leq t_n]}$ and $Z_n=(y_{ij})_{1\leq i \leq j\leq n}$. Then the ESD of $Z_n$ converges almost surely to $\mu$ whose moments are given by
$$\beta_k(\mu) =\left\{ \begin{array}{cc}
\displaystyle\sum_{\sigma \in SS(k)} C_{\sigma} & \text{if } k \text{ is even},\\
0 &  \text{if } k \text{ is odd}.
\end{array}
\right.
$$
Further, the ESD of $W_n$ converges weakly to $\mu$ almost surely (respectively in probability) if \begin{align*}
\frac{1}{n} \sum_{i,j} \ x_{ij}^2[\boldsymbol {1}_{\{|x_{ij}| > t_n\}}] \rightarrow 0, \ \mbox{almost surely (respectively in probability)}.
\end{align*}
\end{cor}
A detailed discussion on our result and the existing results is given in Section \ref{discussion} after the proofs of Theorem \ref{thm:maingk} and Corollary \ref{thm:main}. We also state and prove some new results there which follow directly from Theorem \ref{thm:maingk} and Corollary \ref{thm:main} or follow from the ideas in their proofs. 

%
%


\section{Circuits, words and partitions}\label{Circuits and Words}


\noindent A \textit{circuit} $\pi$ is a function $\pi : \{0,1,2,\ldots,k\}\longrightarrow \{1,2,3,\ldots,n\}$ with $\pi(0)=\pi(k)$. We say that the \textit{length} of $\pi$ is $k$ and denote it by $\ell(\pi)$. Using  circuits, we can express $\Tr(W_n^k)$ as
\begin{align}\label{momentf1usual}
\mathbb{E}\big[\Tr(W_{n}^{k})\big]  = \sum_{\pi:\ell(\pi)=k}\mathbb{E}\big[x_{\pi(0),\pi(1)}x_{\pi(1),\pi(2)}\cdots x_{\pi(k-1),\pi(k)}\big]=\sum_{\pi:\ell(\pi)=k}\mathbb{E}[X_{\pi}],
\end{align}
where $X_{\pi}= x_{\pi(0),\pi(1)}x_{\pi(1),\pi(2)}\cdots x_{\pi(k-1),\pi(k)}$.
For any $\pi$, the (unordered) pairs $\{\pi(i-1),\pi(i)\}$ will be called \textit{edges}. 
When an edge appears more than once in a circuit $\pi$, then $\pi$ is called 
\textit{matched}. 
Any $m$ circuits $\pi_1,\pi_2,\ldots,\pi_m$ are said to be \textit{jointly-matched} if each edge occurs at least twice across all circuits. They are said to be \textit{cross-matched} if each circuit has an edge which occurs in at least one of the other circuits.\\
Two circuits $\pi_1$ and $\pi_2$ of length $k$ are said to be \textit{equivalent} if and only if for all $i,j$, \textbf{as unordered pairs}, 
\begin{align*}
(\pi_1(i-1),\pi_1(i))=(\pi_1(j-1),\pi_1(j))
\Leftrightarrow (\pi_2(i-1),\pi_2(i))=(\pi_2(j-1),\pi_2(j)).
\end{align*} 
For every $k$, the above is an equivalence relation on $\{\pi:\ell(\pi)=k\}$. 
Any equivalence class  of circuits can be indexed by an element of $\mathcal{P}(k)$.
The positions where the edges match are identified by each block of a partition of $[k]$. For example, the partition $\{\{1,3\}, \{2,4,5\}\}$ of $[5]$ corresponds to the equivalent class 
$$\big\{\pi:\ell(\pi)=5\ \mbox{and}\ (\pi(0),\pi(1))=(\pi(2),\pi(3)),\ (\pi(1),\pi(2))=(\pi(3),\pi(4))=(\pi(4),\pi(5))\big\}$$  
Also, an element of $\mathcal{P}(k)$ can be identified with a \textit{word} of length $k$ of letters.  Given a partition, we represent the integers of the same partition block by the same letter, and   the first occurrence of each letter is in alphabetical order and vice versa. For example, the partition $\{\{1,3\}, \{2,4,5\}\}$ of $[5]$ corresponds to the word $ababb$. On the other hand, the word $aabccba$ represents the partition $\{\{1,2,7\},\{3,6\},\{4,5\}\}$  of $[7]$. A typical word will be denoted by $\boldsymbol{\omega}$ and its $i$-th letter as $\boldsymbol{\omega}[i]$. For example, for the word $\boldsymbol{\omega}=abbacac$, $\boldsymbol{\omega}[2]=b$ and the partition is $\{\{1,4,6\},\{2,3\},\{5,7\}\}$. Hence from this point onwards we will use the words \textit{partition} and \textit{word} interchangeably.\\

\noindent \textit{The class $\Pi(\boldsymbol {\omega})$}: The equivalence class of circuits corresponding to a word $\boldsymbol {\omega}$ is denoted by $\Pi(\boldsymbol {\omega})$. So 
\begin{align*}
\Pi(\boldsymbol {\omega})= \big\{\pi: \boldsymbol {\omega}[i]=\boldsymbol {\omega}[j] \Leftrightarrow (\pi(i-1),\pi(i))= (\pi(j-1),\pi(j))\big\}.
\end{align*}
If $\boldsymbol {\omega}[i]=\boldsymbol {\omega}[j]$ , then as ordered pairs,
\[
  (\pi(i-1),\pi(i)) = 
  \begin{cases}
                (\pi(j-1),\pi(j))               & \text{($(C1)$ )   or} \\
                              (\pi(j),\pi(j-1))      & \text{($(C2)$)}.
  \end{cases}
\]
We call the above two situations as $(C1)$ and $(C2)$ respectively.

We now define some classes of words (partitions) that we shall use frequently. 
\begin{definition}
A word $\boldsymbol {\omega}$ is called \textit{even} if each distinct letter in $\boldsymbol {\omega}$ appears an even number of times. This is the same as saying that each block 	of the corresponding partition  has even size. 
The set of all even words (or partitions) of length $2k$ is denoted by $E(2k)$.	
\end{definition}
For example, $ababcc$ is an even word of length 6 with 3 letters. The corresponding partition of $[6]$ is $\{\{1,3\},\{2,4\},\{5,6\}\}$.		
\begin{definition}
A word $\boldsymbol {\omega}$ is called \textit{symmetric} if each distinct letter appears equal number of times in odd and even positions. The set of all such words (or partitions) of length $2k$ is denoted by $S(2k)$.	
\end{definition}
For instance, $abbbba$ is a symmetric word of length $6$ with $2$ distinct letters. Note that a symmetric word is always even. Hence $S(2k) \subset E(2k) \subset \mathcal{P}(2k)$. But every even word is not symmetric. For example, $ababcc$ is even but not symmetric.

Recall the set of partitions $SS(k)$ in definition \ref{ss(2k)}. The words $\boldsymbol {\omega}$ corresponding to $SS(2k)$ will be called \textit{special symmetric} words. For example, the word $aabbaabb$ is a special symmetric word where the corresponding special symmetric partition is $\{\{1, 2, 5, 6\}, \{3, 4, 7, 8\}\}$.

\begin{definition} \label{pure blocks}
Any string of length $m$ $(m>1)$ of  same letter in $\boldsymbol {\omega}$ will be called a \textit{pure block} of size $m$.
\end{definition}
For example, in the word $aabcccbd$, $a$ and $c$ appear in pure blocks of sizes 2 and 3 respectively.

As seen above, the word $aabbaabb$ is a special symmetric word where the corresponding special symmetric partition is $\{\{1, 2, 5, 6\}, \{3, 4, 7, 8\}\}$. 
 For such words, the last new letter appears in pure blocks of even sizes and between two successive appearances of any letter; the other letters that appear, occur even number of times and also an equal number of times in the odd and the even positions. 

We have given a verbal description of the set $SS(2k)$ in Section \ref{mainresults}. For completeness sake, here is a more formal description.  A subset of consecutive natural numbers $\{i,i+1, \ldots,   i+h\}$ is called an \textit{interval} in $\mathbb N$ and denoted by $[i, \ i+h ]$. For example, $[k]$ and $\{2,3,4,5\}=[2, \ 5]$ are intervals in $\mathbb N$. We define an  order  on disjoint  subsets of $\mathbb N$ as follows. For $V_1,V_2\subset \mathbb N$ and $V_1\cap V_2=\emptyset$, we say $V_1<V_2$ if $\min\{j:j\in V_1\}< \min\{j:j\in V_2\}$.  
 For $\sigma  \in \mathcal{P}(2k)$, we  write $\sigma$ as $ \{V_1,V_2,\ldots,V_r\}$  where $V_1,V_2,\ldots,V_r$ are the blocks of the partition $\sigma$ and $V_1<V_2<\cdots <V_r$.  
Now we can describe $SS(2k)$ as follows: 
\begin{enumerate}
\item[(a)] For every $k\in \mathbb N$, the single block partition is $\{1,2, \ldots , 2k\}$ an element of $SS(2k)$.  
 For $k=1$, this is the only element of $SS(2)$.
 \item[(b)] for $k > 1$, a partition  $\sigma = \{V_1 <V_2 < \cdots < V_r\}$ of $[2k]$ belongs to $SS(2k)$ if and only if
  	\begin{enumerate}
\item[(i)] The \textit{last} block  $V_r$ is a union of even sized intervals, say $V_r= I_1 \cup I_2 \cup \cdots\cup I_l$.
\item[(ii)]$|\{[\min\ I_j,\max\ I_{j+1}] \setminus(I_j \cup I_{j+1})\}\cap V_i|$ is even for $1\leq i \leq r-1$ and $1\leq j \leq l-1$.
\item[(iii)] $\sigma \setminus V_r$ can be realized as a special symmetric partition of $[2k]\setminus V_r$.
\end{enumerate}	
\end{enumerate}
The subclass of $SS(2k)$ that has exactly 
$b$ blocks will be denoted by ${SS}_b(2k)$. The corresponding word will have exactly $b$ distinct letters.
For example, the partition $\{\{1,2,5,6\},\{3,4,7,8\},\{9,10\}\}$ of $[10]$ belongs to $SS_3(10)$.  The corresponding special symmetric word is $aabbaabbcc$.
\begin{remark}
 The set $SS(2k)$ \textit{does not} form a lattice under the inclusion ordering. For example, the partitions $\{\{1,2\},\{3,6\},\{4,5\}\}$ and $\{\{1,6\},\{2,3,4,5\}\}$ both belong to $SS(6)$.  However, the largest partition which are smaller than both satisfies
$$\{\{1,2\},\{3,6\},\{4,5\}\} \wedge \{\{1,6\},\{2,3,4,5\}\}=\{\{1\},\{2\},\{3\},\{4,5\},\{6\}\} \notin SS(6).$$
\end{remark}
\begin{definition} 
If $\pi$ is a circuit then any $\pi(i)$  will be called a \textit{vertex}. This vertex is \textit{generating} if $i=0$ or $\boldsymbol {\omega}[i]$ is the first occurrence of a letter in the  word $\boldsymbol {\omega}$ corresponding to $\pi$. All other vertices are \textit{non-generating}.
\end{definition} 

For example, for the word $abc$,  $\pi(0), \pi(1), \pi(2)$ and $\pi(3)=\pi(0)$  are generating.
For the word $aaa$, $\pi(0)$ and $\pi(1)$ are generating.
For the word $abcabc$ $\pi(0), \pi(1), \pi(2)$ and $\pi(3)$ are generating. It so happens that in this case due to the structure of the word, $\pi(3)=\pi(0)$. 
 
Note that the circuits corresponding to a word $\boldsymbol {\omega}$ are completely determined by the generating vertices. 
The vertex $\pi(0)$ is always generating, there is one generating vertex for each new letter in $\boldsymbol{\omega}$. So, if $\boldsymbol {\omega}$ has $b$ distinct letters then the number of generating vertices is $(b+1)$. As seen in the example above, some of these vertices may be identical, depending on the nature of the word. In any case, 
\begin{equation}\label{cardiality of word}
| \Pi(\boldsymbol {\omega})| \leq \mathcal{O}(n^{b+1})\ \ \text{whenever} \ \omega \text{ has } \\   b \ \ \text{distinct letters}.
\end{equation}  
 Now suppose $\boldsymbol {\omega}$ is a word of length $k$ with $b$ distinct letters. Then each repeated letter gives rise to constraints (either $(C1)$ or $(C2)$) which forces some of the $\pi(i)$ to be identical in value, thus forming a partition of the vertices $$V_{\omega}=\{\pi(i):0 \leq i \leq k\}.$$ Note that the number of partition blocks is \textit{at most} $b+1$. 
 For example, the word $abc$ gives three partition blocks,
 $\pi(0)=\pi(3)$,  $\pi(1)$ and $\pi(2)$.
The word $abcabc$ gives three partition blocks,
$\pi(0)=\pi(3)=\pi(6)$,  $\pi(1)=\pi(4)$ and $\pi(2)=\pi(5)$. As two of the generating vertices $\pi(0)$ and $\pi(3)$ are equal, they do not give rise to distinct partition blocks. On the other hand, the word $aabbbb$ gives three partition blocks,
$\pi(0)=\pi(2)=\pi(4)=\pi(6)$,  $\pi(1)$ and $\pi(3)=\pi(5)$ as all the generating vertices in this case are distinct.
   
%

For a word $\omega$ with $b$ distinct letters, suppose $i_1,i_2,\ldots,i_b$ are the positions where new letters made their first appearances in $\boldsymbol {\omega}$. 

Denote by $\mathcal{E}_{i_j}$ the partition block where 
$\pi_{i_{j}}$ belongs. Also let us denote $\mathcal{E}_0$ to be the partition block where $\pi(0)$ belongs. Note that any two such blocks are either equal or disjoint. 
For example for the word $abcabc$, $\mathcal{E}_{0},$
$\mathcal{E}_{1}$, $\mathcal{E}_{2}$ are distinct and 
$\mathcal{E}_{3}=\mathcal{E}_{0}$.

%
	

    Note that  if the number of partition blocks is exactly equal to $b+1$ then 
  $$\displaystyle \lim_{n \rightarrow \infty}\frac{| \Pi(\boldsymbol {\omega})|}{n^{b+1}}=1.$$ 
On the other hand, if the number of partitions blocks is strictly less than $b+1$ then 
$| \Pi(\boldsymbol {\omega})|\leq n^{b}$ and hence  
  $$\displaystyle \lim_{n \rightarrow \infty}\frac{| \Pi(\boldsymbol {\omega})|}{n^{b+1}}=0.$$ 
We shall now investigate when either of this happens.

  \begin{lemma}\label{lem:wig1}
	Let  $\boldsymbol {\omega}$ have  $b$ distinct letters and satisfy
\begin{equation}\label{wigner1}
\frac{|\Pi(\boldsymbol {\omega})|}{n^{b+1}}\to 1 \ \  \text{ as } n \rightarrow \infty.
\end{equation}
 Then the following are true:
	\begin{enumerate}
	\item[(i)] The last new letter of the word appears only in pure blocks of even sizes.  
	\item[(ii)] The successive appearances of the same letter obey the $(C2)$ constraint.
	\item[(iii)] Between successive appearances of each letter, other letters can only appear equal number of times in the odd and even positions.
	\end{enumerate}
	Hence $\boldsymbol {\omega}$ is a special symmetric word.
\end{lemma}	
\begin{proof}
From the discussion before the lemma, there are $(b+1)$ partition blocks, $\mathcal{E}_{0},\mathcal{E}_{i_1},\mathcal{E}_{i_2},\ldots,\mathcal{E}_{i_b}$.
%

Fix any letter $x$ in $\boldsymbol {\omega}$. Then there exists $t\in \{1,2,\ldots,b\}$ such that $x$ first appears at the position $(\pi(i_t-1), \pi(i_t))$. Then we will say that  $x$ is  \textit{associated} with $\pi(i_t)$ and the partition block $\mathcal{E}_{i_{t}}$. Also we can say that $\pi(i_t-1)\in \mathcal{E}_{i_s}$ for some $s<t$.\\


\noindent\textit{Proof of (i)}. Suppose $z$ is the last new letter of $\boldsymbol {\omega}$ and appears at $(\pi(i_b-1),\pi(i_b))$-th position in $\boldsymbol {\omega}$ for the first time. Suppose this $z$ is followed by $x\neq z$. So $x$ is an old letter, say equal to the $j$th new letter of $\boldsymbol {\omega}$ where $j<b$. Then clearly 
either \\
(a) $\pi(i_b)=\pi(i_j)$, that is, $\pi(i_b)\in \mathcal{E}_{i_j}$, or \\
(b) $\pi(i_b)=\pi(i_j-1)$, that is, $\pi(i_b)\in \mathcal{E}_{i_t}$ (where $\pi(i_j-1) \in \mathcal{E}_{i_t}$) for some $t<j$. \\
In any case, $\mathcal{E}_{i_b}$ coincides with $\mathcal{E}_{i_m}$ for some $m\neq b$. As a result $|\Pi(\boldsymbol {\omega})|\leq n^b$ and $\boldsymbol {\omega}$ does not satisfy \eqref{wigner1}. So, $z$ has to be followed by itself, that is, it appears in a pure block.

 Now suppose, this pure block size is odd. 
Consider the last $z$ that appears in this block at $(\pi(i_b+s-1),\pi(i_b+s))$ where $s$ is even. Then
\begin{align*}
&\pi(i_b-1)=\pi(i_b+1)=\pi(i_b+3)= \cdots = \pi(i_b+s-1),\\
&\pi(i_b)=\pi(i_b+2)=\pi(i_b+4)= \cdots = \pi(i_b+s).
\end{align*}
Therefore $\pi(i_b+s) \in \mathcal{E}_{i_b}$. If this $z$ is followed by an old letter then by the same arguments as given earlier, we arrive at a contradiction. 
Hence this block size has to be even and $s$ is odd. 

If $z$ does not appear elsewhere, we are done. So suppose after this block, $z$ appears again next at the $(\pi(t-1),\pi(t))$-th position where $t-(i_b+s)>1$. Then there are two possibilities: either (a) $\{\pi(t-1)= \pi(i_b),   \pi(t)=\pi(i_b-1)\}$
 or (b) $\{\pi(t-1)= \pi(i_b-1),  \pi(t)=\pi(i_b)\}$. As an old letter appears at the $(\pi(t-2),\pi(t-1))$-th position,  
if (a) happens, then by the argument as before, $\mathcal{E}_{i_b}$ coincides with one of the other partition blocks $\mathcal{E}_{0}=\mathcal{E}_{i_0},\ldots,\mathcal{E}_{i_{b-1}}$. This  contradicts \eqref{wigner1}.
So, the only possibility is (b),  
i.e., these two $z$'s satisfy the $(C2)$ constraint. Now we are back to the same situation as that we had for the first appearance of $z$ in the first block of $z$'s. Repeating that argument,  
it follows that 
this $z$ is also followed by an odd number of $z$. Same argument holds for all blocks. 
This completes the proof of (i).\\
 
\noindent \textit{Proof of (ii)}. Now we drop all these $z$'s and consider the reduced word
$\boldsymbol {\omega}'$. During the course of the proof of (i), we have already seen that successive appearances of two $z$'s obey the $(C2)$ constraint and $z$ appears in even pure blocks. Thus $\pi(i_b)$ is the only generating vertex that has been dropped in the process of reduction. Also  all the vertices in the partition block $\mathcal{E}_{i_{b}}$ are dropped. 
Therefore $\boldsymbol {\omega}'$ has $b$ distinct partition blocks (as only one of the partition block is reduced), 
$\mathcal{E}_{i_0},\mathcal{E}_{i_1},\mathcal{E}_{i_2},\ldots,\mathcal{E}_{i_{b-1}}$. 

Now let $y$ be the last new letter in $\boldsymbol {\omega}'$. As $\boldsymbol {\omega}'$ satisfies $\displaystyle \lim_{n \rightarrow \infty} \frac{|\Pi(\boldsymbol {\omega}')|}{n^b}=1$, the argument given earlier for $z$ can be repeated to show that 
$y$ must appear in pure even blocks and any two successive appearances of $y$ must obey the $(C2)$ constraint. 

 Clearly this argument can be repeated sequentially to complete the proof of (ii).\\ 
 
\noindent \textit{Proof of (iii)}. We shall use the following lemma to prove (iii). 
For the moment we assume this lemma and complete the proof of (iii).  
We shall give the proof of Lemma \ref{lemma:result} after that.
\begin{lemma}\label{lemma:result}
Suppose $\boldsymbol {\omega}$ is as stated in Lemma \ref{lem:wig1}. Suppose $x_1,x_2,\ldots,x_m$ appear in between two successive appearances of the same letter, say $y$. Then, all $x_i$ cannot be distinct.
\end{lemma}
Let us focus on the string of letters $z_1,z_2,\ldots,z_m$ in between two successive appearances of $x$ at $(\pi(s),\pi(s+1))$ and $(\pi(s+m+1),\pi(s+m+2))$-th positions.

By Lemma \ref{lemma:result}, there exists $j$ such that $z_j$ has appeared at least twice. 
For any such letter $z_j$, let $p_j$ be the second last position of its appearance (in between the two $x$'s considered). 
Let $p=\max \{p_l:z_l$ has appeared more than once in between the two $x'\text{s}\}$. 
There is some $z_i$ such that $p_i=p$. That is, $z_i$ appears for the second last time at the $p$-th position. Consider the $z_i$ at the $p$-th position and the last $z_i$ (in between the two $x$'s considered). 
By our choice of $p$, no letter can appear more than once in between. Now we invoke Lemma \ref{lemma:result} to conclude that 
there are no other letters in between these $z_i$'s. So these two $z_i$'s appear in consecutive positions, and hence once in an odd position and once in an even position.  
Now we drop this pair of $z_i$'s. If this $z_i$ does not appear elsewhere in the word then the partition block corresponding to $z_i$ is dropped and for the reduced word $\boldsymbol {\omega}'$ with $(b-1)$ distinct letters, $\displaystyle \lim_{n \to \infty}\frac{|\Pi(\boldsymbol {\omega}')|}{n^b}=1$. So we repeat the process. 

If this $z_i$ appears somewhere else in $\boldsymbol {\omega}$, then the reduced word $\boldsymbol {\omega}'$ still has $b$ distinct letters and $\displaystyle \lim_{n \to \infty}\frac{|\Pi(\boldsymbol {\omega}')|}{n^{b+1}}=1$. This is because none of the partition blocks have been dropped or have coincided  in the process of dropping the two $z_i$'s as this pair of $z_i$ are in a  $(C2)$ relation and as a result only one member of a partition block has been omitted. So here again $\boldsymbol {\omega}'$ retains the properties of $\boldsymbol {\omega}$ and we may repeat the process. Continuing this process, 
either there are no letters left in between the $x$'s or all letters that remain are distinct. But the latter 
is not possible due to Lemma \ref{lemma:result}. 
So, we conclude that each distinct $z_j$ appears even number of times in between the two $x$'s, and same number of times in odd and even positions. This completes the proof of  (iii).

Now from (i) and (iii), we conclude that $\boldsymbol {\omega}$ is a special symmetric word.
\end{proof}

\begin{proof}[Proof of Lemma \ref{lemma:result}]  Let $y$ be any letter of $\boldsymbol {\omega}$ with successive occurrences at $(\pi(s),\pi(s+1))$ and $(\pi(s+m+1),\pi(s+m+2))$-th positions. Let $x_1,x_2,\ldots,x_m$ be the letters in between these $y$'s, all distinct. 
Note that $x_i$ appears at $(\pi(s+i),\pi(s+i+1))$-th position. Suppose 
$\pi(s+1) \in \mathcal{E}_{i_j}$. Now there are two cases:\\

\noindent Case 1: Suppose $\pi(s+2)$ belongs to the partition block associated to $x_1$. As $x_1$ and $x_2$ are distinct, $\pi(s+2)$ does not belong to the partition block associated to $x_2$. Therefore, $\pi(s+3)$ belongs to the partition block associated to $x_2$. Again as $x_{3}$ is different from $x_2$, $\pi(s+4)$ belongs to the partition block associated to $x_{3}$. Therefore, repeating this argument for all $x_i$'s, we see that for each $1\leq i \leq m$, $\pi(s+i+1)$ belongs to the partition block associated to $x_i$. Now we know that $\pi(s+1)=\pi(s+m+1)$ since the two appearances of $y$ are in $(C2)$. Therefore $\pi(s+m+1) \in \mathcal{E}_{i_j}$ for some $j$. Also, $\pi(s+m+1)$ belongs to the partition block associated to $x_m$. Thus the partition block associated to $x_m$ coincides with $\mathcal{E}_{i_j}$. Now for $\boldsymbol {\omega}$ to satisfy \eqref{wigner1}, $x_m$ must be the $j$th new letter of $\boldsymbol {\omega}$. Therefore, $\pi(s+m+1)=\pi(i_j)$ and $\pi(s+m)=\pi(i_j-1)$. Hence we have $\pi(s+m) \in \mathcal{E}_{i_{t_1}}$, where $t_1 < j$. Proceeding in this manner we have
\begin{align*}
&\pi(s+m-1) \in \mathcal{E}_{i_{t_2}} \ \ \text{where } \ \ t_2 < t_1; \ \  \pi(s+m-2) \in \mathcal{E}_{i_{t_3}} \ \ \text{where } \ \ t_3 < t_2;\ \ldots;\\ 
&\pi(s+2) \in \mathcal{E}_{i_{t_{m-1}}} \ \ \text{where } \ \ t_{m-1} < t_{m-2}; \ \  \pi(s+1) \in \mathcal{E}_{i_{t_{m}}} \ \ \text{where } \ \ t_{m} < t_{m-1}<j. \\
\end{align*} 
This shows that $\mathcal{E}_{i_j}$ and $ \mathcal{E}_{i_{t_m}}$ coincide for some $t_m \neq j$. So in this case, $| \Pi(\omega)|\leq n^b$, which contradicts  \eqref{wigner1}.\\

\noindent Case 2: Suppose $\pi(s+1)$ belongs to partition block associated to $x_1$. Now we know that $\pi(s+1)=\pi(s+m+1)$ since the two appearances of $y$ are in $(C2)$. As $x_1$ and $x_m$ are distinct, $\pi(s+m+1)$ does not belong to the partition block associated to $x_m$.  Thus, $\pi(s+m)$ belongs to the partition block associated to $x_m$. So, $\pi(s+m-1)$ belongs to the partition block associated to $x_{m-1}$. Therefore, we see that for each $1\leq i \leq m$, $\pi(s+i)$ belongs to the partition block associated to $x_i$. 
 
 We have observed that $\pi(s+m)$ belongs to the partition block associated to $x_m$. We know that $\pi(s+m)=\pi(i_{t_1})$ for some $t_1 \in \{1,2,\ldots,b\}$. Then, $\pi(s+m+1)=\pi(i_{t_1}-1)$. Now recall that $\pi(s+1)= \pi(s+m+1)$ and $\pi(s+1)\in \mathcal{E}_{i_j}$. So, $\pi(s+m+1)\in \mathcal{E}_{i_j}$. Therefore, $\pi(i_{t_1}-1)\in \mathcal{E}_{i_j}$. For the partition block associated to each letter to be distinct, we must have, $\pi(s+m) \in \mathcal{E}_{i_{t_1}}$ where $t_1 > j $.  Proceeding in this manner we have
\begin{align*}
&\pi(s+m-1) \in \mathcal{E}_{i_{t_2}} \ \ \text{where } \ \ t_2 > t_1;\ \pi(s+m-2) \in \mathcal{E}_{i_{t_3}} \ \ \text{where } \ \ t_3 > t_2; \ \ldots;\\ 
&\pi(s+2) \in \mathcal{E}_{i_{t_{m-1}}} \ \ \text{where } \ \ t_{m-1} > t_{m-2};\ \pi(s+1) \in \mathcal{E}_{i_{t_{m}}} \ \ \text{where } \ \ t_{m} > t_{m-1}>j. \\
\end{align*} 
This shows that $\mathcal{E}_{i_j} \cap \mathcal{E}_{i_{t_m}} \neq \phi$ for some $t_m \neq j$. 
 But then $| \Pi(\omega)|\leq n^b$ and it contradicts \eqref{wigner1}. This completes the proof of Lemma \ref{lemma:result}.
\end{proof}
The following corollary follows from Lemma \ref{wigner1}.
\begin{cor}\label{cor:non-ss(k)}
Suppose the word $\boldsymbol {\omega}$ with $b$ distinct letters does not belong to $SS(2k)$ for any $k$. Then $|\Pi(\boldsymbol {\omega})| \leq n^b$. Hence $\displaystyle \lim_{n \rightarrow \infty}\frac{1}{n^{b+1}} | \Pi(\boldsymbol {\omega})|= 0$.
\end{cor}
We now establish  the converse of  Lemma \ref{lem:wig1}. This will identify the words which contribute in the limiting spectral distribution of $W_n$. 
\begin{lemma}\label{lem:wig2}
Suppose $\boldsymbol {\omega}\in SS(2k)$ has  $b$ distinct letters. 
Then, $\displaystyle \lim_{n \rightarrow \infty}\frac{1}{n^{b+1}} | \Pi(\boldsymbol {\omega})|= 1$. 
\end{lemma}
\begin{proof}
We use induction on the number of distinct letters to prove this lemma.
First note that the length of the word is even.
  
When $b=1$, all letters are identical.
In this case $\pi(0)$ and $\pi(1)$ can be chosen in $n \times n=n^2$ ways. Given these vertices, all other vertices have exactly one choice each.
Therefore, $| \Pi(\boldsymbol {\omega})|= n^{2}$ and hence the result is true for $b=1$.

Now suppose, that the result is true for words with $(b-1)$ distinct letters. We shall prove the result for all words with $b$ distinct letters.

By property (i) the last new letter of the word, say $z$, 
appears in pure even blocks. Drop these blocks. 
Now the reduced word $\boldsymbol {\omega}'$ has $(b-1)$ distinct letters. 
For the reduced word 
$\boldsymbol {\omega}'$, 
property (ii) and (iii) are clearly true. 
We now show that $\boldsymbol {\omega}'$ also satisfies property (i) of Lemma \ref{lem:wig1}. Then we can apply our induction hypothesis on $\boldsymbol {\omega}'$. 

Let the last new letter in $\boldsymbol {\omega}'$ be $y$. If $y$ appears only in pure even blocks in $\boldsymbol {\omega}$, then we are done.

Now, suppose that instead,
\textit{there exists a string of $y$'s (from left to right) of odd length in $\boldsymbol {\omega}$}. Under this scenario, we make the following claim:\\

\noindent\textbf{Claim:} There is a string $\boldsymbol {\omega}_1$ in $\boldsymbol {\omega}$ such that
\begin{enumerate}
\item[(i)] It starts with the string of $y$'s of odd length which is followed by a string of $z$'s of even length.
\item[(ii)] It ends with a $y$-string of odd length and cannot be followed by $z$.
\item[(iii)] In between the above two strings there will be strings of $y$ and $z$ of even (may be 0) lengths  in $\boldsymbol {\omega}_1$.
\end{enumerate}
Suppose this result is true. Now after we have dropped the $z$'s from $\boldsymbol {\omega}$, $y$'s have to appear in pure even blocks in $\boldsymbol {\omega}'$. Therefore, $\boldsymbol {\omega}'$ satisfies property (i), (ii) and (iii) of Lemma \ref{lem:wig1}. So, by induction hypothesis, $\displaystyle \lim_{n \to \infty}\frac{| \Pi(\boldsymbol {\omega}')|}{n^b}=1$. The generating vertex of $z$ in $\boldsymbol {\omega}$ has $n$ choices. Therefore, $\displaystyle \lim_{n \rightarrow \infty}\frac{1}{n^{b+1}} | \Pi(\boldsymbol {\omega})|= 1$. This completes the proof of the lemma. 
\end{proof}
Now we prove our claim.\\

\noindent \textit{Proof of the Claim}: Suppose the first pure block of $y$ of odd length (may be of length one) is found. If it is followed by an old letter, say $x$, then between this $x$ and another $x$ just preceding it, there are an odd number of $y$'s, which violates property (iii) of Lemma \ref{lem:wig1}. Therefore, this string of $y$'s can only be followed by $z$ in $\boldsymbol {\omega}$. As $z$ always appears in even pure blocks in $\boldsymbol {\omega}$, this string of $z$'s is of even length. If this string of $z$'s is followed by an older letter say $x$ (not $y$), then the number of $y$'s appearing in between two successive appearances of $x$ is odd which cannot be true because $\boldsymbol {\omega}$ satisfies property (iii) of Lemma \ref{lem:wig1}. So the string of $z$'s mentioned here has to be followed by a string of $y$'s. Then we can have the following two cases:\\

\noindent Case 1: \textit{This string of $y$'s is of odd length}. Then, this string cannot be followed by a string of $z$'s because that would say that there are odd number of $y$'s in between two successive appearances of $z$. So this string has to be followed by an old letter (older than $z$ and $y$), say $x$. Thus we get the sub-word $\boldsymbol\omega_1$ as desired.\\

\noindent Case 2: \textit{This string of $y$'s is of even length}. This string of $y$'s cannot be followed by an older letter say $x$, as then between this $x$ and another $x$ just preceding it, there are odd number of $y$'s which violates property (iii) of Lemma \ref{lem:wig1}. So this string can only be followed by a string of $z$'s. Now as $y$ appears an even number of times in $\boldsymbol {\omega}$, there has to be another pure block of $y$ of odd length (may be of length one). Note that till we obtain such a string,  there can be strings of $y$ and $z$ of even lengths and an older letter say $x$ cannot appear. After we get the string of $y$ of odd length, we see that it cannot be followed by $z$ by the same argument as in case 1. Hence we get the  sub-word $\boldsymbol\omega_1$ as desired.
This completes the proof of the claim.\\

Note that from the above proof it is clear that if we drop all appearances of the last distinct letter of any word $\omega\in SS_{b}(2k)$, then the reduced word belongs to $SS_{b-1}(2j)$ for some $j < k$.  
\section{Proofs of Theorem \ref{thm:maingk} and Corollary \ref{thm:main}}\label{proof of the theorem}
The following elementary result shall be useful.  
See Section 1.2 of \citep{bose2018patterned} for its proof.
\begin{lemma}\label{lem:genmoment}
Suppose $A_n$ is any sequence of symmetric random matrices such that the following conditions hold:
\begin{enumerate}
\item[(i)] For every $k\geq 1$, $\frac{1}{n}\mathbb{E}[\Tr (A_n)^k] \rightarrow \alpha_k$ as $n \rightarrow \infty$.
\item[(ii)] $\displaystyle \sum_{n=1}^{\infty}\frac{1}{n^4}\mathbb{E}[\Tr(A_n^k) \ - \ \mathbb{E}(\Tr(A_n^k))]^4  < \infty$ for every $k \geq 1$.
\item[(iii)] The sequence $\{\alpha_k\}$ is the moment sequence of a unique probability measure $\mu$.
\end{enumerate}
Then $\mu_{A_n}$ converges to $\mu$
weakly almost surely.
\end{lemma}


Condition (ii) of Lemma \ref{lem:genmoment} is known as the \textit{fourth moment condition}. To verify this in all our cases, we shall need the following lemmata. We define  
\begin{align*}
Q_{k,4}^b = | \{& (\pi_1,\pi_2,\pi_3,\pi_4): \ell(\pi_i)=k; \pi_i, 1 \leq i \leq 4 \ \text{jointly- and cross-matched with }\\
 & b \text{ distinct edges or } b \text{ distinct letters across all } (\pi_i)_{1\leq i \leq 4}\}|.
\end{align*}
Here we put a new letter wherever a new edge (or $L-$value) appears across all the circuits $\pi_1,\pi_2,\pi_3$ and $\pi_4$. As the circuits are cross-matched out of the $4k$ places across $(\pi_i)_{1\leq i \leq 4}$, there can be at most $2k$ distinct edges or distinct letters. In $Q_{k,4}^b$ we consider all those quadruple of circuits with $b$ distinct edges and hence $b$ distinct letters.
\begin{lemma}\label{lem:moment}
There exists a constant C, such that,
\begin{equation}
Q_{k,4}^b \leq  C \ n^{b +2}\ .
\end{equation} 
\end{lemma}
 \begin{proof}
First observe that for any circuit $\pi$, if we set aside the first vertex $\pi(0)$, then the number of choices for the generating vertices is $n^b$, where $b$ is the number of distinct letters in $\pi$. Moreoever once all the generating vertices have been chosen, the number of choices for the non-generating vertices is at most one.
This observation will be used repeatedly. 

 Consider all circuits $(\pi_1,\pi_2,\pi_3,\pi_4)$ of length $k$ which are jointly-matched and cross-matched with $b$ distinct letters.
Let the number of new distinct letters appearing in $\pi_i$ be $k_i$, $i \in \{1,2,3,4\}$. So clearly, $k_1+k_2+k_3+k_4=b$.
We begin with the circuit $\pi_1$ and its vertices can be chosen in at most  $n^{k_1+1}$ ways. Now, since the circuits are cross-matched, there is another circuit with which $\pi_1$ shares a letter. So we have the following three cases:\\
 
\noindent{Case 1:} $\pi_1$ shares a letter with only one of the circuits, say $\pi_2$. Then, without loss of generality (since we are dealing with circuits) we can assume that $\pi_2$ begins with the letter it shares with $\pi_1$. Thus, $\pi_2(0)$ and $\pi_2(1)$ both cannot be chosen freely. Hence, having chosen the generating vertices of $\pi_1$, choosing from left to right, the vertices of $\pi_2$
can be chosen in at most $n^{k_2}$ ways. The generating vertices of $\pi_3$ can be chosen in at most $n^{k_3+1}$ ways. Now, since $\pi_4$ does not share any letter with $\pi_1$, it must share at least one letter with either $\pi_2$ or $\pi_3$. Again, we can assume that $\pi_4$ begins with this letter. 
Thus, having chosen the generating vertices of $\pi_1$, $\pi_2$ and $\pi_3$, the vertices $\pi_4(0)$ and $\pi_4(1)$ cannot be chosen freely and the number of choices for all generating vertices of $\pi_4$ is at most 
$n^{k_4}$. Therefore, all vertices of  $(\pi_1,\pi_2,\pi_3,\pi_4)$ can be chosen in at most $n^{(k_1+k_2+k_3+k_4)+2}=n^{b+2}$ ways.\\

\noindent{Case 2:} $\pi_1$ shares a letter with exactly two circuits, say $\pi_2$ and $\pi_3$. Then again we can assume that $\pi_2$ and $\pi_3$ begin with the letters that they share with $\pi_1$. Thus, $\pi_2(0),\pi_2(1),\pi_3(0),\pi_3(1)$ cannot be chosen freely. Hence, choosing from left to right, the generating vertices of $\pi_1(j)$, $\pi_2(j)$, $\pi_3(j)$ 
can be chosen in at most in $n^{(k_1+1)+k_2+k_3}$ ways. Now, $\pi_4$ shares a letter with either $\pi_2$ or $\pi_3$. Again, we can assume that $\pi_4$ begins with this letter.
Hence $\pi_4(0)$ and $\pi_4(1)$ cannot be chosen freely and so the generating vertices of $\pi_4(j)$
can be chosen in at most in $n^{k_4}$ ways. Therefore, all vertices of $(\pi_1,\pi_2,\pi_3,\pi_4)$ can be chosen in at most $n^{(k_1+k_2+k_3+k_4)+2}=n^{b+2}$ ways.\\

\noindent{Case 3:}  $\pi_1$ shares a letter with all the other three circuits. Then, arguing as in the other cases, 
the number of choices is now at most $n^{(k_1+1)+k_2+k_3+k_4}=n^{b+1}$.

This completes the proof of the lemma.	
\end{proof}
The $d_2$ metric 	helps us to deal with situations where entries may not have zero means and/or entries which may require truncation to ensure finiteness of all moments.

Let $F$ and $G$ be two distributions with finite second moment. Then the $d_2$ distance between them is defined as
\begin{align*}
d_2(F,G)= \Big[ \underset{(X \sim  F, Y \sim G)}{\inf} \mathbb{E}[X-Y]^2\Big]^{\frac{1}{2}},
\end{align*}
where $(X  \sim F, Y \sim G)$ denotes that the joint distribution of $(X,Y)$ is such that the marginal distributions are $F$ and $G$ respectively.
It is well-known that if $d_2(F_n,F) \rightarrow 0$ as $n \rightarrow \infty$, then $F_n \xrightarrow {\mathcal{D}} F$.
For a proof see Lemma 1.3.1 in \citep{bose2018patterned}.
This is applicable to ESD due to the following well-known lemma. For a proof of this lemma see Lemma 1.3.2 in \citep{bose2018patterned}.
\begin{lemma}\label{lem:metric}
Let $A$ and $B$ be $n \times n$ symmetric real matrices. Then
\begin{equation}
d_2^2(\mu_A, \mu_B) \leq \frac{1}{n} \Tr(A-B)^2.
\end{equation}
\end{lemma}	
 Now we prove our main theorem.
\begin{proof}[\textbf{Proof of Theorem \ref{thm:maingk}}]\label{proof of thm2}
We break the proof into a few steps.\\

\noindent\textbf{Step 1:} We shall reduce the general case to the case  where all the entries of $Z_n$ have mean 0. To see this, consider the matrix $\widetilde{Z}_n$ whose entries are $(y_{ij}-\mathbb{E}y_{ij})$. Clearly the entries of $\widetilde{Z}_n$ have mean 0. Now
\begin{align}\label{meanzero-noiid}
n\ \mathbb{E}[(y_{ij}- \mathbb{E} y_{ij})^{2k}]= n\ \mathbb{E}[y_{ij}^{2k}] + n \displaystyle \sum_{t=0}^{2k-1} {{2k} \choose {t}}\mathbb{E}[y_{ij}^{t}]\ (\mathbb{E}y_{ij})^{2k-t}.  
\end{align}
The first term of the r.h.s. is equal to $g_{2k}(i/n,j/n)$ by \eqref{gkeven}. For the second term we argue as follows:
\begin{align*}
\text{For } t\neq {2k-1}, \ \ n \ \mathbb{E} [y_{ij}^{t}]\ (\mathbb{E}y_{ij})^{2k-t}& = (n^{\frac{1}{2k-t}}\ \mathbb{E}y_{ij})^{2k-t} \ \mathbb{E}[y_{ij}^{t}]\\
&\overset{n \rightarrow \infty}{\longrightarrow} 0, \ \ \text{ uniformly }  \ \text{ by condition } \eqref{gkodd}.  
\end{align*}
\begin{align*}
\text{For } t={2k-1}, \ \ n \ \mathbb{E} [y_{ij}^{2k-1}]\ \mathbb{E}y_{ij}& = (\sqrt{n} \ \mathbb{E} [y_{ij}^{2k-1}]) \ (\sqrt{n}\  \mathbb{E}y_{ij})\\
&\overset{n \rightarrow \infty}{\longrightarrow} 0, \ \ \text{ uniformly }  \ \text{ by condition } \eqref{gkodd}.
\end{align*}
Hence from \eqref{meanzero-noiid}, we see condition \eqref{gkeven} is true for the matrix $\widetilde{Z}_n$. Similarly we can show that \eqref{gkodd} is true for $\widetilde{Z}_n$. Hence, Assumption A holds for the matrix $\widetilde{Z}_n$.

Now observe that 
\begin{align*}
d_2^2(\mu_{Z_n},\mu_{\widetilde{Z_n}}) & \leq \frac{1}{n} \displaystyle \sum_{i,j} (\mathbb{E}y_{ij})^2\\ 
& \leq n \ (\underset{i,j}{\sup} \ \mathbb{E}y_{ij})^2.\\
&= (\underset{i,j}{\sup} \ \sqrt{n}\ \mathbb{E}y_{ij})^2 \  \overset{n \rightarrow \infty}{\longrightarrow} 0, \ \ \ \ \text{ by condition } \eqref{gkodd}. 
\end{align*}
Hence the LSD of $Z_n$ and $\widetilde{Z}_n$ are same. Hence we can assume the entries have mean zero.

Now we shall use Lemma \ref{lem:genmoment} to prove the first part of the theorem. We will verify the conditions (i), (ii) and (iii) of that lemma using Assumption A and a few other observations made earlier.\\  


\noindent\textbf{Step 2:} We verify the fourth moment condition, i.e.,  (ii) of Lemma \ref{lem:genmoment} for $Z_n$ in this step. 
In particular, we show that 
\begin{equation}\label{4thmoment-noniid}
 \frac{1}{n^{4}} \mathbb{E}[\Tr(Z_n^k) \ - \ \mathbb{E\
 }(\Tr(Z_n^k))]^4\ = \mathcal{O}(n^{-\frac{3}{2}}).
\end{equation}  
 Observe that
 \begin{align}\label{fourthmoment-noniid}
 \frac{1}{n^{4}} \mathbb{E}[\Tr(Z_n^k) \ - \ \mathbb{E}(\Tr(Z_n^k))]^4\ = \frac{1}{n^4} \displaystyle \sum_{\pi_1,\pi_2,\pi_3,\pi_4} \mathbb{E}[\displaystyle \Pi_{i=1}^4 (Y_{\pi_i}\ - \ \mathbb{E}Y_{\pi_i})].
 \end{align}
    If $(\pi_1,\pi_2,\pi_3,\pi_4)$ are not jointly-matched, then one of the circuits has a letter that does not appear elsewhere. 
	Hence by independence, and mean zero assumption, $\mathbb{E}[\displaystyle \Pi_{i=1}^4 (Y_{\pi_i}\ - \ \mathbb{E}Y_{\pi_i})]=0$.
  
	Now, if $(\pi_1,\pi_2,\pi_3,\pi_4)$ are not cross-matched, then one of the circuits say $\pi_j$ is only self-matched. Then, we have $\mathbb{E}[Y_{\pi_j}\ - \ \mathbb{E}Y_{\pi_j}]=0$. So, again we have  $\mathbb{E}[\displaystyle \Pi_{i=1}^4 (Y_{\pi_i}\ - \ \mathbb{E}Y_{\pi_i})]=0$.\\
Thus, we need to consider only circuits $(\pi_1,\pi_2,\pi_3,\pi_4)$ that are jointly- and cross-matched.
Suppose $\pi_i$ has $k_i$ new distinct letters for each $1\leq i \leq 4$ where $k_1+k_2+k_3+k_4=b$. Suppose the $j$th letter appear $s_j$ times across $\pi_1,\pi_2,\pi_3,\pi_4.$
Now the $s_i$'s might be odd or even. Without loss of generality assume that among the $s_i$'s there are $s_{i_1},s_{i_2},\ldots,s_{i_{b_1}}$ which are even and $s_{i_{b_1+1}},s_{i_{b_1+2}},\ldots,s_{i_{b_2}}$ which are odd where $b_1$ and $b_2$ are any two numbers adding up to $b$. Each term then can be written as 
\begin{align*}
 \frac{1}{n^4}  \displaystyle \sum_{b=1}^{2k} n^{-{b_1}} n^{-(b_2-\frac{1}{2})}
 \prod_{j=1}^{b_1} \  g_{s_{i_j},n}(\pi(i_j-1)/n,\pi(i_j)/n) \ \prod_{m=b_1+1}^{b_1+b_2} n^{\frac{b_2-(1-1/2)}{b_2}} \mathbb{E}[y_{\pi(i_{m}-1)\pi(i_{m})}^{s_{i_m}}]. 
 \end{align*}

We note that $g_{s_{i_j},n}   \rightarrow  g_{s_{i_j}}$  for all $1\leq j \leq b_1$. Therefore, the sequence $\|g_{s_{i_j},n}\|$ is bounded by  a constant $M_{j}$. Also as $\frac{b_2-(1-1/2)}{b_2}<1$, by \eqref{gkodd}, we have $n^{\frac{b_2-(1-1/2)}{b_2}} \mathbb{E}[y_{\pi(i_{m}-1)\pi(i_{m})}^{s_{i_m}}]  $ is bounded by $1$ for $n$ large when $b_1+1\leq m \leq b_1+b_2$. Let
$$M^{\prime} = \underset{b_1+b_2=b}{\max} \{M_{t},1:  1 \leq t \leq b_1 \}\ \mbox{ and }\ M_0^{\prime}=  \max \{{M^{\prime}}^b:  1 \leq b \leq 2k \}. 
$$
By Lemma \ref{lem:moment}, we have the total number of such circuits is of the order of $n^{b+2}$.  Therefore we have 
 \begin{align*}
 \frac{1}{n^{4}} \mathbb{E}[\Tr(Z_n^k) \ - \ \mathbb{E}(\Tr (Z_n^k))]^4 
& \leq M_0^{\prime} \displaystyle \sum_{b=1}^{2k} \frac{1}{n^{b+3\frac{1}{2}}} n^{b+2} \\
& = \mathcal{O}(n^{-\frac{3}{2}}).   
 \end{align*}
This completes  the  proof of \eqref{4thmoment-noniid}.
 By Lemma \ref{lem:genmoment} and \eqref{4thmoment-noniid}, it is enough to show  that for every $k\geq 1$, 
$\displaystyle \lim_{n\to\infty}\frac{1}{n}\mathbb{E}[\Tr(Z_n)^{k}]$ exists and is given by $\beta_k( \mu^{\prime})$ for each $k \geq 1$.\\

\noindent\textbf{Step 3:} In this step we verify the first moment condition, (i) of Lemma \ref{lem:genmoment} for $Z_n$. Now from \eqref{momentf1usual} and using the fact that $\mathbb{E}(y_{ij})=0$,  we have 	
\begin{align}\label{momentnoniid1}
\lim_{n \rightarrow \infty}\frac{1}{n}\mathbb{E}[\Tr(Z_n)^{k}]& =\displaystyle\lim_{n \rightarrow \infty}\frac{1}{n} \sum_{\pi:\ell(\pi)=k}\mathbb{E}[Y_{\pi}] \nonumber \\
 &= \displaystyle\lim_{n \rightarrow \infty} \displaystyle \sum_{b=1}^k \frac{1}{n} \sum_{\underset{\text{with b distinct letters}}{\omega \ \text{matched}}} \sum_{\pi \in \Pi(\boldsymbol{\omega})} \mathbb{E}(Y_{\pi})  
\end{align}

Let $\boldsymbol {\omega}$ be a word with $b$ distinct letters and let $\pi \in \Pi(\boldsymbol {\omega})$. Suppose the first appearance of the letters of $\boldsymbol {\omega}$ are at the $i_1,i_2, \ldots,i_b$ positions. Thus, the $j$th new letter appears at the $(\pi(i_j-1),\pi(i_j))-$th position for the first time. Recall the partition blocks $\mathcal{E}_{0},\mathcal{E}_{i_1},\mathcal{E}_{i_2},\ldots,\mathcal{E}_{i_b}$ as defined in Section \ref{Circuits and Words}.

From the proof of Lemma \ref{lem:wig1} and Lemma \ref{lem:wig2}, we know that all the $\mathcal{E}_{i_j} \ (0 \leq j \leq b)$ are distinct if and only if $\boldsymbol{\omega}\in SS_b(2k)$. 

So write \eqref{momentnoniid1} as 	
\begin{align}\label{momentnoniid2}
\lim_{n \rightarrow \infty}\frac{1}{n}\mathbb{E}[\Tr(Z_n)^{k}]
&=  \displaystyle\lim_{n \rightarrow \infty} \displaystyle \sum_{b=1}^k  \Big[\frac{1}{n} \sum_{\omega \in SS_b(k)} \sum_{\pi \in \Pi(\boldsymbol{\omega})}\ \mathbb{E}(Y_{\pi}) + \frac{1}{n} \sum_{\underset{\boldsymbol {\omega} \text{ with b letters}}{\omega \notin SS(k)}}  \sum_{\pi \in \Pi(\boldsymbol{\omega})}\ \mathbb{E}(Y_{\pi})\Big] .\nonumber \\
& = T_1+T_2.
\end{align}

Clearly $T_1$ is the term involving all the special symmetric partitions. This will be shown to contribute positively to the limit. The sum of contributions of all other partitions is $T_2$ and will be shown to go to 0 as $n \to \infty$.

For each $j \in \{1,2,\ldots , b\}$ denote $(\pi(i_j-1),\pi(i_j))$ as $(t_j,l_j)$. Clearly $t_1=\pi(0)$ and $l_1=\pi(1)$. It is easy to see that each distinct $(t_j,l_j)$ corresponds to each distinct letter in $\boldsymbol {\omega}$.
Let $S$ be the collection of representatives of each of the distinct $\mathcal{E}_{l_j}$'s and $\mathcal{E}_{t_1}$. Clearly, by \eqref{cardiality of word}, $|S|\leq (b+1)$.
%

Let $\boldsymbol {\omega}\in SS_b(k)$. Then by Lemma \ref{lem:wig2}, $|S|=b+1$. 
Suppose the $j$th new letter appear $s_j$ times in $\boldsymbol {\omega}$. Clearly all the $s_j$ are even. So the total contribution of this $\boldsymbol {\omega}$ to $T_1$ in \eqref{momentnoniid2} is as follows:
\begin{align}
\frac{1}{n} \displaystyle \sum_{\underset{1 \leq j \leq b}{(t_j,l_j)}} \ \displaystyle \prod_{j=1}^b \ \mathbb{E}\Big[y_{t_j l_j}^{s_j}\Big] 
= \frac{1}{n^{b+1}} \displaystyle \sum_{\underset{1 \leq j \leq b}{(t_j,l_j)}} \displaystyle \prod_{j=1}^b g_{s_j,n}\big(t_j/n,l_j/n\big).\label{finitesum-ss2k}
\end{align}
Next observe that if a sequence of bounded Riemann integrable functions, say $f_n$ converges uniformly to a function $f$ and a sequence of finite measures, say $\nu_n$ converges weakly to a measure $\nu$, then $$\int f_n \ d\nu_n \rightarrow \int f \ d\nu.$$

From this observation it is clear that for any sequence of bounded Riemann integrable function $f_n(x_1,x_2,\ldots,x_{b+1})$ on $[0,1]^{b+1}$, that converges uniformly to $f(x_1,x_2,\ldots,x_{b+1})$, as $n \rightarrow \infty$,
$$\frac{1}{n^{b+1}}\displaystyle \sum_{j_1,\ldots, j_{b+1}=1}^n f_n(j_1/n,j_2,n, \ldots, j_{b+1}/n) \rightarrow \int_{[0,1]^{b+1}}f(x_1,x_2,\ldots,x_{b+1}) \ dx_1dx_2\cdots dx_{b+1}.$$
Now as $|S|=b+1$ for $\boldsymbol {\omega} \in SS_b(2k)$, as $n \rightarrow \infty$, \eqref{finitesum-ss2k} becomes 
\begin{equation}\label{limit-ss2k}
\displaystyle \int_{[0,1]^{|S|}} \ \prod_{j=1}^b g_{s_j}\big(x_{t_j},x_{l_j}\big) \ dx_S,
\end{equation} 
where $dx_S= \prod_{i \in S} dx_i$ denotes the $|S|-$dimensional Lebesgue measure on $[0,1]^{|S|}$.

We split the investigation of $T_2$ into two cases. \\

\noindent\textbf{Case 1.}  Suppose $\boldsymbol {\omega}$ is an even word with $b$ distinct letters but is not special symmetric. Then the contribution to $T_2$ of \eqref{momentnoniid2} can be calculated as in \eqref{finitesum-ss2k}. But now note that  $|S|\leq b$. Hence in this case as $n \rightarrow \infty $, the contribution of this word $\boldsymbol {\omega}$ is 0.\\

\noindent\textbf{Case 2.}  $\boldsymbol {\omega} \notin E(2k)$. Suppose $\boldsymbol{\omega}$ contains $b_1$ distinct letters that appears even number of times and $b_2$ number of distinct letters that appears odd number of times and $b=b_1+b_2$. Without loss of generality, we can assume for each $\pi \in \Pi(\boldsymbol {\omega})$, without loss of generality $s_{j_p}$, $1\leq p \leq b_1$ to be even and $s_{j_q}$, $b_1+1\leq q \leq b_1+b_2$ to be odd. Hence the contribution of this $\boldsymbol {\omega}$ to $T_2$ in \eqref{momentnoniid2} is as follows:
\begin{align}\label{finitesum-nonss2k}
\frac{1}{n}   n^{-{b_1}} n^{-(b_2-\frac{1}{2})}
\sum_{\underset{1 \leq j \leq b}{(t_j,l_j)}} \prod_{p=1}^{b_1} \  g_{s_{j_p},n}\big(t_{j_p}/n,l_{j_p}/n\big) \ \prod_{m=b_1+1}^{b_1+b_2} n^{\frac{b_2-(1-1/2)}{b_2}} \mathbb{E}\Big[y_{t_{i_m}l_{j_m}}^{s_{j_m}}\Big] \nonumber\\
= \frac{1}{n^{b_1+b_2+\frac{1}{2}}}  \displaystyle   
\sum_{\underset{1 \leq j \leq b}{(t_j,l_j)}} \prod_{p=1}^{b_1} \  g_{s_{j_p},n}\big(t_{j_p}/n,l_{j_p}/n\big) \ \prod_{m=b_1+1}^{b_1+b_2} n^{\frac{b_2-(1-1/2)}{b_2}} \mathbb{E}\Big[y_{t_{i_m}l_{j_m}}^{s_{j_m}}\Big].
\end{align}
For $n$ large, $n^{\frac{b_2-(1-1/2)}{b_2}} \mathbb{E}[y_{t_{i_m}l_{j_m}}^{s_{j_m}}]<1$ for any $b_1+1\leq m\leq b_1+b_2$. Now as $|S|\leq b$,  \eqref{finitesum-nonss2k} contributes 0 as $n \rightarrow \infty$.

For any partition $\sigma\in SS_b(2k)$ let $\{V_1, \ldots ,V_b\}$ be its partition blocks. Then from \eqref{momentnoniid2} and \eqref{limit-ss2k}, we have 
\begin{align}\label{wignernoniid}
\lim_{n \rightarrow \infty}\frac{1}{n}\mathbb{E}[\Tr(Z_n)^{2k}]& =    \sum_{b=1}^{k}\displaystyle \sum_{\sigma \in SS_b(2k)} \displaystyle \int_{[0,1]^{|S|}} \ \prod_{j=1}^b g_{|V_j|}(x_{t_j},x_{l_j}) \ dx_S
 \end{align}
where $dx_S= \prod_{i\in S} dx_i$ denotes the $|S|-$dimensional Lebesgue measure on $[0,1]^{|S|}$.

We also note that $\displaystyle\lim_{n \rightarrow \infty}\frac{1}{n}\mathbb{E}[\Tr(Z_n)^{2k+1}]=0$ for any $k\geq 0$.
This completes the proof of the first moment condition.\\

\noindent \textbf{Step 4:} We prove the uniqueness of the measure in this step. We have obtained 
\begin{align*}
\gamma_{2k}= &\lim_{n \rightarrow \infty}\frac{1}{n}\mathbb{E}[\Tr(W_n)^{2k}]
 \leq \displaystyle \sum_{\sigma \in SS(2k)} M_{\sigma}
 \leq \displaystyle \sum_{\sigma \in \mathcal{P}(2k)} M_{\sigma} 
 =\alpha_{2k}.
\end{align*}
As $\{\alpha_{2k}\}$ satisfies Carleman's condition, $\{\gamma_{2k}\}$ also does so. Now using Lemma \ref{lem:genmoment}, we see that there exists a measure $ \mu^{\prime} $ with moment sequence $\{\gamma_{2k}\}$ such that $\mu_{Z_n}$ converges weakly almost surely to $ \mu^{\prime} $. 

This completes the proof of part (a).\\

\noindent\textbf{Step 5:} We prove part (b) of the theorem in this final step. To see this, observe that 
 \begin{align}\label{d2inequality}
 d_2^2(\mu_{W_n},\mu_{Z_n}) \leq \frac{1}{n} \sum_{i,j} \ x_{ij}^2[\boldsymbol {1}_{[|x_{ij}| > t_n]}].
 \end{align}
Now if this $\{t_n\}$ also satisfies condition \eqref{truncation}, then using \eqref{d2inequality} and (a) we can say that the ESD of $W_n$ converges to $ \mu^{\prime}$ almost surely (respectively in probability). 

This proves part (b).

To complete the proof, it remains to verify the condition for the limit to be semi-circular. We define any  word of length $2k$ to be \textit{Catalan} if it is pair-matched and at the same time non-crossing. This corresponds to the partitions in $SS_k(2k)$.

First note that if $g_{2k}=0$ for $k>1$, then by \eqref{wignernoniid} it is only the Catalan words that contribute to the sum. Let $\omega$ be a Catalan word with the last new letter appearing at the $(\pi(i_k-1),\pi(i_k))-$th position. Then the generating vertex for that letter is $\pi(i_k)$. Therefore, going by the notation mentioned in the beginning of this proof and as well as those in the proof of Lemma \ref{lem:wig2}, $i_k$ does not appear in any of the pairs $(t_j,l_j)$ for $j<k$. So, the contribution  $p(\omega)$ of $\omega$ to $\gamma_{2k}$ is as follows:
\begin{align}\label{Catalan}
p(\omega)= \int_{[0,1]^{k+1}} g_2(x_{t_1},x_{l_1})\cdots g_2(x_{t_k},x_{l_k}) \ dx_S. 
\end{align}
Since $x_{l_k}$ does not appear in any of the other factors of the integrand, we can integrate w.r.t $x_{l_k}$ to get 
\begin{align*}
p(\omega)&=\int_{[0,1]^{k}} g_2(x_{t_1},x_{l_1})\cdots \int_{[0,1]}g_2(x_{t_k},x_{l_k}) \ dx_{i_k} \ dx_{S \setminus \{x_{l_k}\}}\\
&= c \ \int_{[0,1]^{k}} g_2(x_{t_1},x_{l_1})\cdots g_2(x_{t_{k-1}},x_{l_{k-1}}) \ dx_{S \setminus \{x_{l_k}\}}, \ \ \  \text{as } \int_{[0,1]} g_2(x,y)\ dy \ = c.
\end{align*} 
Now dropping the last new letter from the word $\omega$, the reduced word $\omega'$ is also a Catalan word. Hence following the same argument we have that \eqref{Catalan} becomes $c^k$ which is independent of $\omega$. Hence for any Catalan word the contribution to $\gamma_{2k}$ is same. Therefore, the limit is semi-circular.

Now suppose that the limit is semi-circular and without loss assume that it has variance 1. Note that then the fourth moment equals $2$. Define 
$$f(x)=\int_{0}^1 g_2(x,y)dy.$$
Then $$\gamma_2=\int_0^1 f(x)dx=1.$$
Then from equation (\ref{limit-ss2k}),  the fourth moment is given by 
\begin{eqnarray*}2&=&\gamma_4\\
 &=&\int_0^1 g_4(x,y)dxdy+ 2\int_0^1 g_2(x_1, x_2)g_2(x_1, x_3)dx_1dx_2dx_3\\
&=& \int_0^1  g_4(x,y)dxdy+2\int_0^1 f^2(x)dx\\
& \geq & \int_0^1  g_4(x,y)dxdy+2\big(\int_0^1 f(x)dx\big)^2\\
& \geq & \int_0^1  g_4(x,y)dxdy+2.
\end{eqnarray*}
Clearly then from the above $f(x)=1$ for all $x$ and $g_4\equiv0$. 

Now note that from earlier calcuations, since $f\equiv 1$, the contribution of each Catalan word (of any order) equals 1. Since this contribution already gives the moments of the semi-circular, the contribution from other words vanish.  As a consequence inductively $g_{2k}=0$ for all $k >1$. 
The proof of the theorem is now complete.
\end{proof}
\begin{remark}
The upper bound used to prove Carleman's condition in Step 4 of the above proof can be strengthened further because $SS(2k)\subset S(2k)\subset E(2k)$. 
\end{remark}
\begin{remark}
When we look at the case where $g_{2k}=0$ for all $k>1$, we know from  \eqref{wignernoniid} that only Catalan words contribute to the sum. But without any further condition on $g_2$, the contribution of different Catalan words may be unequal. For example, let $\omega_1=aabbcc$ and $\omega_2=abccba$ be two Catalan words of length 6. By \eqref{limit-ss2k}, the contribution for  $\omega_1$ is  $$p(\omega_1)=\int_{[0,1]^4} g_2(x_1,x_2)g_2(x_1,x_3)g_2(x_1,x_4) \ \prod_{i=1}^4 dx_i$$ while the contribution for  $\omega_2$ is  $$ p(\omega_2)=\int_{[0,1]^4} g_2(x_1,x_2)g_2(x_2,x_3)g_2(x_3,x_4) \ \prod_{i=1}^4 dx_i$$ Obviously $p(\omega_1)\neq p(\omega_2)$ in general.
 Also, it can be verified that under the assumption that $g_{2k}=0$ for all $k>1$, the condition $\int_{[0,1]} g_2(x,y)\ dy $ is constant is necessary for the limit to be semi-circular.
\end{remark}

\begin{proof}[\textbf{Proof of Corollary \ref{thm:main}.}]
We know that  $\{x_{ij}; 1 \leq i \leq j\leq n\}$ are i.i.d. for every fixed $n$. Then $\{y_{ij}; 1 \leq i \leq j\leq n\}$ are also i.i.d. for every fixed $n$. From condition (i) of the corollary, clearly \eqref{gkodd} and \eqref{gkeven} are satisfied with $g_{2k,n}\equiv C_{2k}$ on $[0,1]$. Therefore, $g_{2k}\equiv C_{2k}$ on $[0,1]$ and $W_n$ satisfies condition (ii) of Assumption A. Having observed this, condition (ii) of the corollary implies condition (iii) of Assumption A. Thus, from Theorem \ref{thm:maingk}, the ESD of $Z_n$ converges to a probability measure $\mu$. 

From \eqref{wignernoniid}, we see that only the special symmetric words contribute in the limiting moment sequence. Also as $g_{2k}\equiv C_{2k}$ on $[0,1]$, the moments of $\mu$ are given by 
$$\beta_k(\mu) =\left\{ \begin{array}{cc}
\displaystyle\sum_{\sigma \in SS(k)} C_{\sigma} & \text{if} \ k  \ \text{is even},\\
0 &  \text{if}\  k \ \text{is odd}.
\end{array}
\right.
$$

Now suppose further that $\frac{1}{n}\displaystyle \sum_{i,j} \ x_{ij}^2[\boldsymbol {1}_{[|x_{ij}| > t_n]}]$ converges to 0 almost surely (respectively in probability). Then by Theorem \ref{thm:maingk}, the ESD of $W_n$ converges to $\mu$ almost surely (respectively in probability).
\end{proof}
 \begin{remark} 
 Suppose $F$ is the distribution function of a symmetric infinitely divisible distribution with all moments finite and cumulant seqeunce $\{D_k\}_{k \geq 1}$. Then due to infinite divisibility, for every $n$, we can find i.i.d. random variables $\{y_{i,n}: 1\leq i \leq n1\}$ with distribution $H_n$, such that $\sum_{i=1}^{n}y_{i,n}$ converges in distribution to $F$ (see page 766, Characterization 1 in \citep{bose2002contemporary}). Moreover, it can be easily verified that 
the above weak convergence holds if
\begin{equation}\label{eq:momentconv}n \mathbb{E}[y_{i,n}^k]\rightarrow D_k.
\end{equation}
  Now let $\{x_{ij,n}: \ i\leq j\}$ be i.i.d. with distribution $H_n$, for every fixed $n$ such that (\ref{eq:momentconv} holds. Then from the above discussion it is clear that these variables satisfy \eqref{ckeven} and  \eqref{ckodd} with $t_n=\infty$ and $C_{2k}=D_{2k}$. 
  Now if moments of $F$ satisfy Carleman's condition, then the variables $\{x_{ij,n}: \ i\leq j\}$ satisfy the assumptions in Corollary \ref{thm:main}. Thus the ESD of $W_n$ with entries $\{x_{ij,n}: \ i\leq j\}$ as described here converges almost surely to the symmetric probability distribution $F$ which is identified by $\{D_{2k}\}_{k \geq 1}$.
This gives a class of non-trivial matrices whose almost sure LSD exists by an application of Corollary \ref{thm:main}.
\end{remark}
\begin{remark}\label{momentgenfn}
If condition (ii) of Corollary \ref{thm:main} is replaced by the condition that the sequence $\{0, C_2, 0,C_4, 0,\ldots\}$ is  the cumulant sequence  of a probability distribution $G$ whose moment generating function has a positive radius of convergence around 0, then the result still holds. To see this, suppose $X \sim G$. Suppose $Y$ is a random variable whose moments are as as follows:
\begin{align*}
\mathbb{E}(Y^{2k-1}) = 0 \ \mbox{ and }\ 
 \mathbb{E}(Y^{2k})  = \displaystyle \sum_{\sigma \in SS(2k)} C_{\sigma} \ \ \ \text{ for each } k \geq 1.
\end{align*}
Then observe that
\begin{align}\label{Y}
\mathbb{E}(Y^{2k})= \displaystyle \sum_{\sigma \in SS(2k)} C_{\sigma}
\leq \displaystyle \sum_{\sigma \in E(2k)} C_{\sigma}=  \mathbb{E}(X^{2k})
\end{align}
and
\begin{align*}
 0 \leq M_Y(t)  = \displaystyle \sum_{k=0}^{\infty} \frac{t^k}{k!} \mathbb{E}(Y^k)
  = \displaystyle \sum_{k=0}^{\infty} \frac{t^{2k}}{(2k)!} \mathbb{E}(Y^{2k}) \leq \displaystyle \sum_{k=0}^{\infty} \frac{t^{2k}}{(2k)!}  \mathbb{E}(X^{2k})= M_X(t).
\end{align*}
So if $M_X(t)$ has a positive radius of convergence around 0, then 
$M_Y(t)$ has a positive radius of convergence around 0. This implies that the distribution of $Y$ is uniquely determined by its moments, and everything else follows as in the proof of Theorem \ref{thm:maingk}.
\end{remark}
 \section{Discussion}\label{discussion}

In this section we connect Theorem \ref{thm:maingk} and Corollary \ref{thm:main} to Results \ref{result:semi-circle}--\ref{result:Zhu}. We also prove some extensions and new results. In Sections \ref{i.i.d}, we show how Result \ref{result:semi-circle} follows from Corollary \ref{thm:main} and then provide an extension in Result \ref{res:variance} to matrices that have independent entries but with unequal variances. 
In Section \ref{heavy-tailed}, we deduce Result \ref{Heavy-tailed entries, BenArous} by using Theorem \ref{thm:maingk} and a truncation argument. In Section \ref{sparse}, we first show how Result \ref{result:sparse} follows from Corollary \ref{thm:main}. Result \ref{imhomoge}
is a new result on the existence of the almost sure LSD of the adjacency matrix of a non-homogeneous  Erd\H{o}s-R\'{e}nyi graph in the sparse regime. In Section \ref{triangular iid}, we deduce Result \ref{result:Zakharevich} from Corollary \ref{thm:main}. We also give an alternative description of the special symmetric words in terms of coloured rooted trees in Lemma \ref{lem:wig3},  connecting the limiting moments in Corollary \ref{thm:main} to those given in \citep{zakharevich2006generalization}. In Section \ref{graphons}, we show how we can generalize the graphon approach and thereby deduce Result \ref{result:Zhu}. In Section \ref{variance prf}, we state and prove Result \ref{res:variance profile} for matrices with a variance profile. In Section \ref{band and block} we state and prove Results \ref{res:band1}, \ref{res:band2}, \ref{res:block} to generalize results given in \citep{casati1993wigner} and \citep{zhu2020graphon}. 
We end with a few simulations in Section \ref{simulations} to show the varying nature of the LSD.

\subsection{Result \ref{result:semi-circle} and a non-iid extension}\label{i.i.d}
 We first consider $W_n$ with i.i.d. entries and then with only independent entries.\\
  
\noindent \textbf{(a) I.I.D. case}: Suppose $W_n=(x_{ij}/\sqrt{n})$, where $\{x_{ij}\}$ are i.i.d. with distribution  $F$,  mean zero and variance one.  

Let $t_n=n^{-1/3}$. Then $ t_n\sqrt{n} \rightarrow \infty$ as $n \rightarrow \infty$ and
\begin{align*}
\displaystyle \lim_{n \rightarrow \infty} n \ \mathbb{E}\bigg [\bigg(\frac{x_{ij}}{\sqrt{n}}\bigg)^2 \boldsymbol{1}_{[|x_{ij}/\sqrt{n}| \leq t_n]}\bigg] \ =  \ 1=  C_2.
\end{align*}
Also, for any $k>2$,
\begin{align*}
n\ \mathbb{E}\bigg [\bigg(\frac{x_{ij}}{\sqrt{n}}\bigg)^{k} \boldsymbol{1}_{[|x_{ij}/\sqrt{n}| \leq t_n]}\bigg] \ &
= n\ \mathbb{E}\big [{(x_{11}/\sqrt{n})}^{(k-2)} \ {(x_{11}/\sqrt{n})}^{2} \boldsymbol{1}_{[|x_{11}| \leq t_n  \sqrt{n}]}\big]\\
& \leq n \frac{t_n^{(k-2)}}{n} \mathbb{E} \big [{x_{11}}^{2} \boldsymbol{1}_{[|x_{11}| \leq t_n  \sqrt{n}]}\big]\\
& \leq t_n^{(k-2)}\\
& =  (n^{- \frac{1}{3}})^{k-2} \ \rightarrow \ 0 \ \ \text{ as } n \rightarrow \infty.
\end{align*}
Now for any $ t > 0$, 
\begin{align*}
\frac{1}{n}\displaystyle \sum_{i,j} \ \big(x_{ij}/\sqrt{n}\big)^2[\boldsymbol {1}_{[|x_{ij}/\sqrt{n}| > t_n]}]= & \frac{1}{n^2}\displaystyle \sum_{i,j} \ x_{ij}^2[\boldsymbol {1}_{[|x_{ij}| > t_n \sqrt{n}]}]\\
& \leq \frac{1}{n^2}\displaystyle \sum_{i,j} \ x_{ij}^2[\boldsymbol {1}_{[|x_{ij}| > t]}] \ \ \text{for all large} \ n,\\
& \overset{a.s.}{\longrightarrow} \ \mathbb{E}\big[ x_{11}^2[\boldsymbol {1}_{[|x_{11}| > t]}]\big] \ \ \text{for all large} \ \ n.  
\end{align*}
As $\mathbb{E}[x_{11}^2]=1$, taking $t$ to infinity, the above limit is 0 almost surely. 

So without loss we may assume that all moments of $F$ are finite.
Let $G_n$ be the distribution of $X/ \sqrt{n}$ for each $n$ where $X\sim F$.
So the $k$th moment of $G_n$ equals $\mu_n(k)= \frac{\beta_k(F)}{n^{k/2}}$ for $k\geq 1$. Thus $n\mu_n(2)=\beta_2(F)=1$  for all $n$. Also, for $k>2$, $n\mu_n(k)=\frac{\beta_k(F)}{n^{k/2-1}}$. As $F$ has all moments finite, we have  $C_2=1$ and $C_{2k}=0$ for all $k>1$.

Hence $W_n=(x_{ij}/\sqrt{n})_{1 \leq i,j \leq n},$
where $x_{ij}=x_{ji}, \ 1 \leq i<j \leq N$ satisfy the assumptions of Corollary \ref{thm:main}. Therefore the ESD of $W_n$ converges almost surely to $\mu$ whose moments are given by
\begin{align*}
\beta_{2k}(\mu) =  \displaystyle \sum_{\sigma \in SS(2k)} C_{\sigma}.
\end{align*}
But for any $\sigma=\{V_1,V_2,\ldots,V_b\}\in SS(2k)$, $C_{\sigma}=0$ if there is any $j$ such that $|V_j|\neq 2$. Therefore, we must have $b=k$ and $|V_i|=2$ for all $1 \leq i \leq k$. Hence $\sigma$ is a pair-partition and in particular, $\sigma$ is a non-crossing pair-partition of $[2k]$ as $\sigma$ belongs to $SS(2k)$ also.
Thus 
\begin{align}\label{momentf1finite}
\beta_{2k}(\mu)  = \displaystyle \sum_{\sigma \in SS(2k)} C_{\sigma}  = \displaystyle \sum_{\sigma \in NC_2(2k)} 1
 = \frac{1}{k+1} {{2k} \choose k}. 
\end{align}
Therefore $(\beta_{2k})_{k \geq 1}$ are the Catalan numbers, and hence the LSD of $W_n=(x_{ij})/\sqrt{n}$ is the  semi-circular distribution. Thus we get Result \ref{result:semi-circle} as a special case of Theorem \ref{thm:maingk}.\\

\noindent \textbf{(b) Independent entries}: We now drop the identically distributed condition and also assume unequal variances for the entries. This provides an extension of the result in  \citep{bai2010spectral} for the equal variance case. 
\begin{result}\label{res:variance}
   Consider the Wigner matrix $W_n$ with entries $\{\frac{x_{ij}}{\sqrt{n}} : 1\leq i \leq j \leq n\}$ that are independent and satisfy the following conditions:
\begin{enumerate}
\item[(i)] $\mathbb{E}x_{ij}=0$ and $\mathbb{E}[x_{ij}^2]= {\sigma_{ij}}^2$.
\item[(ii)] $\sigma_{ij}$ satisfy the following:
\begin{align}\label{variance}
\underset{i}{\sup}\ \displaystyle \bigg|\frac{1}{n}\sum_{j=1}^n {\sigma^2}_{ij} -1\bigg| \rightarrow 0 \ \ \  \text{ as } n \rightarrow \infty.
\end{align}
\item[(iii)] $\displaystyle \lim_{n \rightarrow \infty} \frac{1}{n^2} \displaystyle \sum_{i,j} \mathbb{E}\big[x_{ij}^2]\boldsymbol{1}_{[|x_{ij}|>\eta \sqrt{n}]}\big]=0$ for every $\eta>0$.
\end{enumerate}
Then the almost sure LSD of $W_n$ is the semi-circular distribution. 
\end{result}

\begin{proof}
Note that for every $\eta>0$,
\begin{equation}
\sup_{i} \bigg| \frac{1}{n}\sum_{j}\mathbb{E}\bigg[x_{ij}\boldsymbol{1}_{[|x_{ij}|\leq \eta \sqrt{n}]} -\mathbb{E}\big[x_{ij}\boldsymbol{1}_{[|x_{ij}|\leq \eta \sqrt{n}]} \big] \bigg]^2 \ - \ 1 \bigg| \longrightarrow 0 \ \ \ \text{ using (ii) and (iii)}.
\end{equation}

Again using (iii) and the truncation step (Step 1) in the proof of Theorem 2.9 in \citep{bai2010spectral}, we can assume without loss of generality the entries of $W_n$ to be $x_{ij}\boldsymbol{1}_{[|x_{ij}|\leq \eta_n \sqrt{n}]} -\mathbb{E}\big[x_{ij}\boldsymbol{1}_{[|x_{ij}|\leq \eta_n \sqrt{n}]}\big] $ for some decreasing sequence $\eta_n \downarrow 0$. 
 

First observe that for the word $aa$, the contribution to the moment sequence is 1. This is because 
\begin{align*}
\big|\frac{1}{n} \underset{i_0,i_1}{\sum}\  \frac{1}{n} {\sigma^2}_{i_0,i_1} -1 \big|\ \leq \underset{i_0}{\sup}\ \displaystyle |\frac{1}{n}\sum_{i_1=1}^n {\sigma^2}_{i_0i_1} - 1| \rightarrow 0 \ \ \  \text{ as } n \rightarrow \infty. 
\end{align*} 
We shall prove by induction on the length of the word that each Catalan word contributes 1 to the limit. Towards that suppose all Catalan words of length $2(k-1)$ contribute 1 to the moment.

Now suppose  $\omega$ is a Catalan word of length $2k$. Recall the notation used in the proof of Theorem \ref{thm:maingk}. Using those notation, clearly $|S|=k+1$ and each distinct letter corresponds to the pair $(t_j,i_j), \ (1 \leq j \leq k)$. As $\omega$ is Catalan, $i_k$ appears only once in the sum. Therefore for this word $\omega$, 
\begin{align}\label{independent}
\frac{1}{n^{k+1}} \underset{i_0,i_1,\ldots,i_k}{\sum}\ \prod_{j=1}^k {\sigma^2}_{t_j,i_j} 
= \frac{1}{n^k} \underset{i_0,i_1,\ldots,i_{k-1}}{\sum}\ \prod_{j=1}^{k-1} {\sigma^2}_{t_j,i_j} \ \big(\frac{1}{n}\sum_{i_k} {\sigma^2}_{t_k,i_k}-1\big)\ + \frac{1}{n^k} \underset{i_0,i_1,\ldots,i_{k-1}}{\sum}\ \prod_{j=1}^{k-1} {\sigma^2}_{t_j,i_j}.
\end{align}  
The second term of the r.h.s. of the above equation goes to 1 as $n$ goes to $\infty$ by the induction hypothesis. Observe that the factor $$\big(\frac{1}{n}\sum_{i_k} {\sigma^2}_{t_k,i_k}-1\big) \rightarrow 0\ \ \ \text{ by } \eqref{variance}.$$ Also note that as $\omega$ is Catalan we can write the first term as 
\begin{align*}
\frac{1}{n}\sum_{i_0} \prod_{j=1}^{k-1}\big( \frac{1}{n}\sum_{i_j}{\sigma^2}_{t_j,i_j} \big) \big(\frac{1}{n}\sum_{i_k} {\sigma^2}_{t_k,i_k}-1\big).
\end{align*}
Moreover by \eqref{variance}  $\big(\frac{1}{n}\sum_{i_{j}} {\sigma^2}_{t_{j},i_{j}}\big)$ is bounded for each $1 \leq j \leq k-1$. Thus the first term of the r.h.s. of  \eqref{independent} goes to 0 as $n \rightarrow \infty$. Therefore every Catalan word contributes 1 in the limit.  

Now, suppose $\omega$ is a non-Catalan word with $b$ distinct letters which appear $s_1,s_2,\ldots,s_b$ times. So, $|S|\leq b$. Then the contribution for this word is as follows: 
\begin{eqnarray*}
\frac{1}{n^{k+1}} \underset{\text{distinct elements of}\  S}{\sum} \prod_{j=1}^b \mathbb{E}[x^{s_j}_{t_j,i_j}]
& \leq& \frac{(\eta_n \sqrt{n})^{2k-2b}}{n^{k+1}} \underset{\text{distinct elements of} \ S}{\sum} \prod_{j=1}^b {\sigma^2}_{t_j,i_j}\\
& =& \frac{{\eta_n}^{2k-2b}}{n^{b+1}}\underset{\text{distinct elements of} \ S}{\sum} \prod_{j=1}^b {\sigma^2}_{t_j,i_j}\\
& \leq &\frac{{\eta_n}^{2k-2b+2}}{n^{b}}\underset{\text{distinct elements of} \ S\setminus \{i_m\}}{\sum} \underset{j\neq m}{\prod_{j=1}^b} {\sigma^2}_{t_j,i_j} \ \  (\text{as}\  \ |S|\leq b) \\
&=& \mathcal{O}\big({\eta_n}^{2k-2b+2}\big) \ \ \text{since} \ \   \big(\frac{1}{n}\sum_{i_{j}} {\sigma^2}_{t_{j},i_{j}}\big) \ \ \text{is bounded}\\
&\to & 0 \ \ \text{since}\ \ b\leq k\ \ \text{and}\ \ \eta_n \downarrow 0.
\end{eqnarray*}
Thus such words do not contribute in the limit. Therefore, the moments of the limiting EESD are those of the semi-circular distribution.

Again, using the above technique along with Lemma \ref{lem:moment}, condition (ii) of Lemma \ref{lem:genmoment} can be verified. Hence the ESD of $W_n$ converges weakly to the semi-circular distribution almost surely.\\
\end{proof}



\subsection{Result \ref{Heavy-tailed entries, BenArous}: heavy tailed entries}\label{heavy-tailed}
 Suppose $F$ is an $\alpha$-stable distribution ($0<\alpha <2$), i.e., there exists a slowly varying function $L$ such that $$\mathbb{P}[|x_{ij}|\geq u]= \frac{L(u)}{u^{\alpha}}.$$ Now we consider $W_n=(x_{ij}/a_n)$ where $\{x_{ij}\}$ are i.i.d. with distribution $F$ and $a_n= \inf\{u: \mathbb{P}[|x_{ij}|\geq u]\leq \frac{1}{n}\}$. \citep{arous2008spectrum}  proved the existence of LSD of $W_n$  using the method of Stieltjes transform. We show how our theorem  may be used to  give an alternative proof. 
 
For a fixed constant say, $B$, consider the matrix $W_n^B$ whose entries are $\frac{x_{ij}}{a_n}\boldsymbol {1}_{[|x_{ij}|\leq B a_n]}$. Then we have the following: 
 \begin{enumerate}
 \item[(a)]  For every fixed $B\in \mathbb{N}$, $W_n^B$ satisfies assumptions of Corollary \ref{thm:main}. Hence there exists a probability measure $\mu_B$ which is the  weak limit of the ESD of $W_n^B$ almost sure. That is, for each fixed $B$ and $\epsilon>0$, for $n$ large enough,
 \begin{align}\label{finitelevel}
 d_1(\mathbb{E}[\mu_{B}],\mathbb{E}[\mu_{W_n^B}]) \leq \epsilon/3, \ \ \ \ \text{ almost surely}
 \end{align}
 where  $d_1$ (see Lemma 2 in \citep{arous2008spectrum}) is an appropriate metric giving weak convergence of distributions.
 \item[(b)] By Theorem 2.2 in \citep{arous2008spectrum}, for every $\epsilon>0$, there exists $B(\epsilon)$ and $\delta(\epsilon,B)>0$ such that for $n$ large enough 
 \begin{align}\label{closeness}
 \mathbb{P}\big[d_1(\mu_{W_n},\mu_{W_n^B})>\epsilon\big] \leq \exp(-\delta(\epsilon,B)n).
 \end{align}
 Hence for every $\epsilon>0$, there exists $B(\epsilon)$ and $\delta(\epsilon,B)>0$ such that for $n$ large enough
 \begin{align}\label{closeness-exp}
 d_1(\mathbb{E}(\mu_{W_n}),\mathbb{E}(\mu_{W_n^B})) \leq \exp(-\delta(\epsilon,B)n).
 \end{align}
 Hence using \eqref{finitelevel} and \eqref{closeness-exp}, we have 
 that $\mu_B$ is $d_1-$ Cauchy. As the space of all distributions is complete with respect to this metric, $\mu_B$ converges to  a probability measure $\tilde{\mu}$, say.
 \item[(c)] Next observe that 
 \begin{align*}
 \mathbb{P}\big[d_1(\mu_{W_n},\tilde{\mu})>\epsilon\big] &
 \leq \mathbb{P}\big[d_1(\mu_{W_n},\mu_{W_n^B})>\epsilon/3\big]+\mathbb{P}\big[d_1(\mu_{W_n^B},\mu_{B})>\epsilon/3\big]+\mathbb{P}\big[d_1(\tilde{\mu},\mu_{B})>\epsilon/3\big]
 \end{align*}
 Choosing  $B$ large enough and then taking $n$ to $\infty$, we have that the r.h.s. of the above inequality is very small.
\end{enumerate}  
Hence we conclude that $\mu_{W_n}$ converges weakly to $\tilde{\mu}$ in probability. This yields Result \ref{Heavy-tailed entries, BenArous}.

\subsection{Result \ref{result:sparse} and a non-homogenous extension}\label{sparse}

\noindent \textbf{(a) Homogeneous case}: Let $G_n= Bin(1,{p}_n)$ where $n{p}_n \rightarrow \lambda>0$ as $n\to\infty$. Consider the  Wigner matrix   
\begin{align*}
W_n=(x_{ij})_{1\leq i, j \leq n},
\end{align*}
where $\{x_{ij}; \ 1 \leq i\leq j \leq n\}$ are i.i.d. $G_n$. 
It is easy to see that the assumptions of Corollary \ref{thm:main} holds with $t_n=\infty$ and $C_{2k}= \lambda $ for all $k$.  
Hence  by Corollary \ref{thm:main}, the ESD  of $W_n$ converges almost surely to $\mu$ whose moments are given by \begin{align}\label{momentf1sparse}
\beta_{2k}(\mu)= \displaystyle \sum_{\sigma \in SS(2k)} C_{\sigma}  = \displaystyle \sum_{\sigma \in SS(2k)} \lambda^{|\sigma|}=\sum_{b=1}^k |SS_b(2k)|\lambda^b.
\end{align}

Bauer and Golinelli \citep{bauer2001random} considered the matrix $M_n$ with entries as in $W_n$ but  the diagonal entries being 0.
First note that the LSD of $M_n$ and $W_n$ are identical. To see this observe,
\begin{align*}
d_2^2(\mu_{W_n},\mu_{M_n}) & \leq \frac{1}{n} \displaystyle \sum_{i=1}^n x_{ii}^2.
\end{align*}
 As $x_{ii}\sim  Bin(1,p_n)$ for each fixed $n$ and independent across $i$, $X_n=\sum_{i=1}^n x_{ii}^2$ follows $Bin(n,p_n)$ where $n p_n\rightarrow \lambda$. This implies that $\frac{X_n}{n} \rightarrow 0$ almost surely. Hence the r.h.s. of the above inequality goes to 0 almost surely as $n\rightarrow \infty$. 

Recall the expression \eqref{Bauer-exprss} of moments given earlier for the limit of the EESD from \citep{bauer2001random}.
 Now since the limit of the EESD and almost sure limit of the ESD must be equal, the two expressions \eqref{momentf1sparse} and \eqref{Bauer-exprss}, must be identical. Since both expressions are polynomials in $\lambda$, we must have 
 $$|SS_b(2k)|= I_{k,b}\ \ \text{for all}\ \ b.$$
 this yields Result \ref{result:sparse}.
 Later in Lemma \ref{lem:wig3}, we shall relate $SS_b(2k)$ to certain trees which are differently constructed compared to \citep{bauer2001random}.\\

\noindent \textbf{(b) Non-homogeneous case}: Now consider the non-homogeneous Erd\H{o}s-R\'{e}nyi graph in the sparse regime.
\begin{result}\label{imhomoge}
Suppose there is a sequence of bounded Riemann integrable symmetric functions $p_n:[0,1]^2\longrightarrow [0,1]$ such that $n p_n$ converges uniformly to a function $p$. 
Consider the  Wigner matrix   
\begin{align*}
W_n=(x_{ij})_{1\leq i, j \leq n},
\end{align*}
where $\{x_{ij}, \ 1 \leq i\leq j \leq n\}$ are
 such that $n \mathbb{E}[x_{ij}^{2k}]= p_n(i/n,j/n)$.
 Then the ESD of $W_n$ converges waekly almost surely to a symmetric probability measure $\mu^{\prime}$ whose even moments are given by 
 \begin{equation}\label{eq:inhomosparse}\beta_{2k}(\mu^{\prime})= \sum_{\pi \in SS(2k)} \lambda^{|\pi|}= \sum_{b=1}^k |SS_b(2k)|\lambda^b. 
\end{equation}  
 In particular, if $\int_0^1p(x,y)dx=\lambda$ for all $y$, then $\mu^{\prime}=\mu_{bg}$. \end{result}
 
 \begin{proof}
 It is easy to see that $W_n$ satisfies Assumption A with $g_{2k,n}= p_n$ and $g_{2k}=p$ for all $k\geq 1$ and $t_n= \infty$.
Therefore, from Theorem \ref{thm:maingk} the ESD of $W_n$ converges weakly almost surely to 
$\mu^{\prime}$ say. Thus the LSD exists and (\ref{eq:inhomosparse}) holds. 

Now consider the special case where, $\int_{0}^1 p(x,y) \ dy = \lambda$. Then from \eqref{limit-ss2k}, each word in $SS_b(2k)$ contributes $\lambda^b$ to the $2k$th moment of $\mu^{\prime}$. Hence this moment is given by  
\begin{align*}
\beta_{2k}(\mu^{\prime})= \sum_{\pi \in SS(2k)} \lambda^{|\pi|}= \sum_{b=1}^k |SS_b(2k)|\lambda^b.
\end{align*}
  As these moments determine the distribution uniquely, we have $\mu^{\prime}=\mu_{bg}$ (the limit in the homogeneous case).
\end{proof}

  \begin{remark}
	We could of course start with numbers  $p_{i,j,n} \in [0,1]$, $1 \leq i \leq j \leq n$, for each fixed $n$. Then we can create a sequence of continuous functions 
$p_n$ on $[0,1]^2$ such that $$p_n(i/n,j/n)=p_{i,j,n} \ \ \ \ \ \text{ for every } 1 \leq i \leq j \leq n.$$

Now, assume that the functions $n p_n$ converge uniformly to the function $p$ on $[0,1]^2$. Then we can conclude the convergence of the ESD as discussed above.

It can be verified that the condition  $\int_{0}^1 p(x,y) \ dy = \lambda$ is equivalent to the condition $$\underset{i}{\sup} |\sum_{j=1}^n p_{ij,n}-\lambda| \longrightarrow 0.$$
\end{remark}

  \subsection{Result \ref{result:Zakharevich}: triangular IID, colored rooted trees and $SS(2k)$}\label{triangular iid}
 Consider the set up of Result \ref{result:Zakharevich}. It is easy to see that $W_n$ satisfies assumptions of Corollary \ref{thm:main}. Hence the ESD of $W_n$ converges to $\mu$ almost surely. Since this almost sure LSD must agree with the limit of the EESD stated in Result \ref{result:Zakharevich}, we wish to explore the relation between our moment formula and that given in \citep{zakharevich2006generalization} in terms of certain colored rooted trees.

A colored rooted tree is a graph with no cycles with one distinguished vertex as the root, and each vertex has a colour that signifies certain properties. 

The next lemma claims that these colored rooted trees and our special symmetric partitions $SS(2k)$ are in one-to-one correspondence. As a consequence, the limiting moments derived in the proof of Corollary \ref{thm:main} are the same as the one derived in Theorem 2 of \citep{zakharevich2006generalization}. We shall also see how the construction of such trees help us to extend the idea of graph homomorphism densities and thereby to generalize  Result \ref{result:Zhu}.  		
	
\begin{lemma}\label{lem:wig3}
Suppose $\boldsymbol{\omega}\in SS_b(2k)$. Also suppose that each letter appears $2k_1,2k_2,\ldots,2k_b$ times respectively in $\boldsymbol{\omega}$. Then there is a coloured rooted tree corresponding to $\boldsymbol{\omega}$ with $(k+1)$ vertices and $(b+1)$ distinct colours $a_0,a_1,\ldots,a_b$ with the following properties:
\begin{enumerate}
\item [(a)] The root is of  colour $a_0$ and there are exactly $k_i$ vertices of colour $a_i$, $1 \leq i \leq b$.
\item [(b)] If two vertices are of same colour then their parents are also of same colour.
\item [(c)] vertices with the same colour are at the same distance from the root. 
\end{enumerate}
Also, for every such tree with $k+1$ vertices and $b+1$ distinct colours there is a unique word that belongs to $SS_b(2k)$ and vice-versa.

In particular, the \textit{Catalan number} $\{C_{k}\}_{k \geq 1}$ counts the number of coloured rooted trees with $k+1$ vetrices, each vertex with a distinct colour. 
\end{lemma} 
\begin{proof}
Suppose $\boldsymbol{\omega}\in SS_b(2k)$ such that each letter appears $2k_1,2k_2,\ldots,2k_b$ times respectively. Let $i_1,i_2,\ldots,i_b$ be the positions where the distinct  letters made their first appearance in $\boldsymbol{\omega}$. We have already seen in the proof of Lemma \ref{lem:wig2} that for such a word $\boldsymbol{\omega}$ there are exactly $b$ distinct partition block associated to each distinct  letter (containing its generating vertex) and one partition block for $\mathcal{E}_{i_0}$ which contains $\pi(0)$. We assign the colour $a_j$ to the partition block $\mathcal{E}_{i_j}$ for $1 \leq j \leq b$. Therefore, we have $(b+1)$ distinct colours and $(b+1)$ distinct partition blocks. Now we begin constructing a tree from left to right.

Create a root and colour it $a_0$. The first letter in $\boldsymbol{\omega}$ has coordinates $(\pi(0),\pi(1))$. For this we create a child of the root and colour it $a_1$. Then for each successive appearance of this letter which are in $(C1)$ relation with the first appearance, we place a further child (of colour $a_1$), moving from left to right. The letters which are in $(C2)$ relation with the first appearance, are ignored. 

 When the second distinct letter of $\boldsymbol{\omega}$ appears at $(\pi(i_2-1),\pi(i_2))$ position for the first time, there are two possibilities: $\pi(i_2-1)=\pi(0)$ or $\pi(i_2-1)=\pi(1)$. In the first case, we create a further child of the root, place it to the right of the previous children and colour it $a_2$. In the second case, we create a  child for the rightmost $a_1$-colored child and colour it $a_2$. Then for each successive appearance of the letter, if it obeys the $(C1)$ relation with the first appearance,  we create a child with colour $a_2$ for the same $a_1$ colored child, using the same principle. We continue inductively with all other letters in the same manner until all letters are exhausted. Clearly, by construction this gives us a tree with a root of colour $a_0$.


For example consider the word $aaabbacc$. Then $k=4$ and $b=3$. Hence we obtain a tree with $5$ vertices and four colors is as follows:
\begin{center}
\begin{tikzpicture}
     [sibling distance=5em,scale=.9,auto=center,every node/.style={circle,fill=blue!20}] 
    \node {$a_0$}
    child{node {$a_1$}}
    child{node {$a_1$}
    child{node {$a_2$}}
   }
   child{node {$a_3$}}
   ;
    
    \end{tikzpicture}

\end{center}
 Now if $x$ is the $j$th distinct letter of $\boldsymbol {\omega}$, then it appears $2k_j$ times in $\boldsymbol {\omega}$ out of which $k_j$ times the coordinates of $x$ are in $(C1)$ relation with its first appearance (including the first appearance). So there are $k_j$ vertices of the colour $a_j$. So property (a) holds for this tree.


Suppose (b) does not hold for the tree constructed. Then there are two vertices of the same colour say, $a_m$ such that their parents are of different colours, say $a_{p_1}$ and $a_{p_2}$ ($p_1 \neq p_2$). Also these two vertices of colour $a_m$ correspond to the appearance of the $m$-th new letter in positions say, $(\pi(s-1),\pi(s))$ and $(\pi(t-1),\pi(t))$ ($s \neq t$) in  $\boldsymbol {\omega}$ where $\pi(s)=\pi(t)=\pi(i_m)$. By the construction process, we get $\pi(s-1)=\pi(i_{p_1})$ and $\pi(t-1)=\pi(i_{p_2})$. But we know that $\pi(s-1)=\pi(t-1)$. Therefore the two partition blocks $\mathcal{E}_{i_{p_1}}$ and  $\mathcal{E}_{i_{p_2}}$ coincide which cannot be true by Lemma \ref{lem:wig2}. So there cannot be such vertices of same colour with parents of different colours. That is, (b) is true for $\boldsymbol {\omega}$.

Now, vertices of colour $a_1$ can appear only as chlidren of the root according to the construction of the tree. Suppose, property (c) is true for all colours $a_i$ where $1\leq i \leq j-1$. Now there are $k_j$ vertices of the colour $a_j$. By the construction we see that either all of these $k_j$ vertices appear as children of the root or they appear as children of the vertices of the colour $a_t$ where $t<j$. If all $k_j$ vertices are children of the root, we have nothing to prove and (c) holds for the colour $a_j$. In the other case, all of these $k_j$ vertices  appear as children of the vertices of the colour $a_t$ and as $t<j$ all vertices of colour $a_t$ are at the same distance from the root. Therefore, all vertices of colour $a_j$ also are at the same distance from the root. Hence by induction, we have that (c) is true for the tree corresponding to $\boldsymbol {\omega}$. 
Thus for any word $\boldsymbol {\omega}\in SS_b(2k)$, the $(k+1)$ vertices and $(b+1)$ distinct colours satisfy properties (a), (b) and (c).

Further, for any two distinct words $\boldsymbol {\omega}_1$ and $\boldsymbol {\omega}_2$  in $SS_b(2k)$, the above process of construction yield two distinct coloured rooted tree with $(k+1)$ vertices and $(b+1)$ distinct colours each with properties (a), (b) and (c).

Now, suppose we have a  coloured rooted tree with $(k+1)$ vertices and $(b+1)$ distinct colours with properties (a), (b) and (c). We need to show that there is a word in $SS_b(2k)$ corresponding to this tree. 

As there are $(b+1)$ distinct colours $a_0,a_1,a_2,\ldots,a_b$, we can associate to each colour $a_j$ a generating vertex $\pi(i_j)$ (and hence to the partition block $\mathcal{E}_{i_j}$) for $1 \leq j \leq b$ and we associate $a_0$ with $\pi(0)$ and hence the partition block $\mathcal{E}_{i_0}$. We traverse the tree from left to right in the depth-first way starting at the root. 
For every step downward we get a vertex of colour $a_j$, we add the $j$th distinct letter to the word and for every step upward to a vertex of colour $a_t$ we add  the $t$-th distinct letter to the word. We repeat this process for all the branches of the tree, left to right.

The first vertex appearing after the root is of colour $a_1$ which creates the letter $a$. 

If there are no further children of this vertex then we  come back to $a_0$ (root) and add the letter $a$ to obtain the partial word $aa$.  We then move to the next right branch. 

If instead $a_1$ has at least one child, then we traverse one step down, to the next vertex which is given the color $a_2$ and we add the second distinct letter to $a$ to obtain the partial word $ab$. 

We repeat the above process on all branches, always moving depth first and left to right. Note that eventually we traverse upward on the rightmost branch and come back to the root. In this process we end up with a word  $\boldsymbol {\omega}$ of length $2k$ since each edge has been traversed twice. Also there are $b$ distinct letters as for each of the colours $a_1,a_2,\ldots,a_b$, a distinct letter has been added to form the word. 
We now wish to verify that $\boldsymbol {\omega}\in SS_b(2k)$. 
So we need to verify (i), (ii) and (iii) of Lemma \ref{lem:wig1}.

Due to the construction, among all the vertices of color $a_b$,   the first one corresponds to the first appearance of the last distinct letter of $\boldsymbol {\omega}$. Clearly, no vertex of  colour $a_b$ has any children. So whenever we go downward to a vertex of colour $a_b$, in the next step we must come back to its parent vertex. Hence property (i) of Lemma \ref{lem:wig1} is satisfied.
 
We now verify property (iii).  Consider two successive appearances of the same letter say, $x$. If they are side by side, there is nothing to verify. 
 
Now suppose they have some other letters in between. Suppose $x$ is the $j$th distinct  letter of $\boldsymbol {\omega}$. Then there are two ways in which we have got the first $x$: (a) the first $x$ appeared while we were going down to the vertex of colour $a_j$ which has a further child or (b) while we were coming upward from a vertex of colour $a_j$. 
 
In case (a), during the construction, we had gone down that branch and reached the end of this sub-tree and each time added a letter to the word for each vertex. Having reached the end, we had started coming upward and added those letters in the reverse order to the word to reach the vertex of colour $a_j$ we started with. Because of property (c) we cannot get the next $x$ before we reach this vertex. Hence we see that in between these two successive $x$'s each letter has been added even number of times. Therefore we have property (iii) of Lemma \ref{lem:wig1} for $\boldsymbol {\omega}$. 
 
In case (b), we have got the first $x$ while coming upward from a vertex of colour $a_j$ (to a vertex of color $a_t$ say). Observe from the properties (b) and (c) of the tree that the next $x$ can only occur while going downward to another vertex of colour $a_j$ whose parent must have the  colour $a_t$. By property (c), same coloured vertices occur only at the same level of the tree. So to come to this next $x$, we have to keep traversing the tree from left to right until we reach a vertex of colour $a_j$ from $a_t$. Clearly in this process we pass each intermediate vertex exactly two times, once going up and once going down. Hence as before each letter can appear only even number of times in between these two successive $x$'s. Moreover, since each pair of intermediate letters also satisfies this property, it is easy to infer that the two $x$'s are in $(C2)$ relation. 

 Therefore, $\boldsymbol {\omega}\in SS_b(2k)$. Finally, it is also clear that different colored rooted trees yield different words. 
 
Now it is clear from above that each Catalan word corresponds to a colored rooted tree with $k+1$ vertices and $k+1$ distinct colours. Hence the Catalan number counts the number of colored rooted tree with $k+1$ vertices each vertex with a distinct colour.

 This completes the proof of the lemma.
\end{proof}
Since Lemma \ref{lem:wig3} provides a description of $SS(2k)$ in terms of colored rooted trees, 
we can express the $2k$th moment of the LSD of $W_n$ 
in terms of these trees. Let $T_{2k}^b(k_1,k_2,\ldots,k_b)$ denote the number of coloured rooted trees  with $(k+1)$ vertices and $(b+1)$ distinct colours which satisfy properties (a), (b) and (c) of Lemma \ref{lem:wig3}.
 
Therefore, under the assumption of  Corollary \ref{thm:main}  the $2k$th moment of the LSD of $W_n$ is given by
\begin{align*}
\displaystyle\lim_{n \rightarrow \infty}\frac{1}{n}\mathbb{E}[\Tr(W_n^{2k})]
  = \displaystyle \sum_{\sigma \in SS(2k)} C_{\sigma}  
  = \displaystyle \sum_{b=1}^k \sum _{k_1+k_2+\cdots+k_b=k}T_{2k}^b(k_1,k_2,\ldots,k_b)\prod_{i=1}^b C_{2k_i}.
\end{align*}
Thus we obtain Result \ref{result:Zakharevich} from Theorem \ref{thm:maingk}.

\subsection{Result \ref{result:Zhu}: graphons and beyond}
\label{graphons}

Theorem 3.2 of  \citep{zhu2020graphon}, stated as Result \ref{result:Zhu} earlier, considers   Wigner matrix with independent entries $\{a_{ij}\}$ (subject to symmetry) having mean zero and $\mathbb{E}[|a_{ij}|^2]=s_{ij}$ whose distribution do not depend on $n$. In proving the convergence of the ESD the author used the concept of graphons and homomorphism density to express the limiting moments. 
  First we prove a more general result using Lemma \ref{lem:wig3}.  We connect the special symmetric words and hence the moments of the limiting distribution of Theorem \ref{thm:maingk} with a more general class of trees (see Lemma \ref{lem:wig3} in Section \ref{triangular iid}). 

Let us generalize the concept of homomorphism density for the coloured rooted trees that appear in Lemma \ref{lem:wig3}. Define, 
for each $k$, a \textit{graphon sequence}  
$M_{2k,n}$ that takes the value $$n \mathbb{E}[x_{ij,n}^{2k}\boldsymbol{1}_{[|x_{ij,n}|\leq t_n]}]=g_{2k,n}(i/n,j/n)\ \ \text{ on }I_i \times I_j (1 \leq i,j \leq n )$$
where $t_n$ is a sequence as defined in Assumption A.

For each word in $SS(2k)$ with $b$ distinct letters, we have a coloured rooted tree $T^{\prime}$ as described in Lemma \ref{lem:wig3}. Denote its vertex set by $V=:\{0, \ldots k\}$, enumerated by first appearances, left to right and depth first. Each vertex is painted with a colour from the colour set $C=:\{c_0, \ldots , c_b\}$, say. 
Let $E$ be the edge set of $T^{\prime}$. Observe that there can be many edges whose vertices have a fixed pair of colours $c_i$ and $c_j$. 
Enumerate $E$ as follows:  
\begin{align*}
E=&\cup_{0\leq i < j\leq b}E(i,j)\\
E(i,j)=& \ \big\{(v_1, v_2) \in E:  v_1 < v_2\ \ \text{are coloured}\ \  c_i \ \ \text{and} \ \ c_j \ \ \text{respectively}\big\}, 0\leq i < j \leq b.
\end{align*}
 
The homomorphism density $t(T, H_n)$ in \eqref{zhu-density} is now extended to 
\textit{generalized homomorphism density} as 
\begin{align}\label{homomorphismdensity}
t(T^{\prime}, \{M_{2k,n}\}) = \int_{[0,1]^{b+1}} \underset{0\leq i<j\leq b} {\prod_{(i,j)\in E}} g_{2|E(i,j)|,n}(x_{i},x_j) \ \prod_{0 \leq i \leq b} dx_i.
\end{align}

Now we state and prove the following result:
\begin{result}\label{res:graphons}
Suppose $W_n$ is the $n \times n$ Wigner matrix with independent entries $\{x_{ij,n};i \leq j\}$ that satisfies \eqref{gkeven}, \eqref{gkodd} and \eqref{truncation}. Suppose that 
\begin{equation}\label{gen-homomorphism density}
t(T^{\prime}, \{M_{2k,n}\})\ \ \text{ converges for all coloured rooted trees } T^{\prime}.
\end{equation}   Then the ESD of $W_n$ converges weakly almost surely (or in probability) to a distribution whose odd moments are $0$ and the $2k$th moment is given by $$\sum_{T\ \ \text{colored rooted trees}}\lim t(T, \{M_{2k,n}\})$$
provided these moments determine a unique probability distribution.
\end{result}
\begin{proof}
In order to establish the first moment condition, first observe that the words that do not belong to the set $SS(2k)$, do not contribute to  the limiting $2k$th moment. Also from Lemma \ref{lem:wig3}, we know that each word in $SS_b(2k)$ corresponds to a coloured rooted tree with $b$ distinct colours. Hence the contribution for each such word (or tree) is $\displaystyle \lim_{n \rightarrow \infty} t(T, \{M_{2k,n}\})$ (as this limit exists). Thus we get the first moment condition.
The fourth moment condition can be verified in the same manner as in Step 2 of the proof of Theorem \ref{thm:maingk}. 

Finally, as  these moments determine a unique probability distribution, using Lemma \ref{lem:genmoment}, we conclude that the ESD of $W_n$ converges weakly almost surely (or in probability) to a symmetric  distribution $\mu$ whose odd moments are $0$ and the $2k$th moment is given by $$\beta_{2k}(\mu)=\sum_{T\ \ \text{colored rooted tree}}\lim t(T, \{M_{2k,n}\}).$$
\end{proof}

%
    

 \noindent Now we shall show how Result \ref{result:Zhu} follows from Result \ref{res:graphons}. Consider the matrix $\tilde{W}_n$ as defined in Result \ref{result:Zhu}. Thus the variables $\{a_{ij}\}$ as defined in Result \ref{result:Zhu} satisfy the three conditions  \eqref{boundedvar}, \eqref{lindeberg} and \eqref{treecondition} as in \citep{zhu2020graphon}. Note that there is a sequence $\eta_n$ decreasing to $0$ such that \eqref{lindeberg} continues to hold with $\eta$ replaced by $\eta_n$. Using the truncation step (step 1) in the proof of Theorem 2.9 in \citep{bai2010spectral}, it follows that it is enough to consider the entries of $\tilde{W}_n$ to be bounded by $\eta \sqrt{n}$. Now, $\lim t(T^{\prime}, \{M_{2k,n}\})=0$ for all trees with less than $k+1$ colours and $\lim t(T^{\prime}, \{M_{2,n}\})= \lim t(T^{\prime}, H_n)$ exists finitely for all trees with $k+1$ colours. Therefore from Result \ref{res:graphons}, we conclude that the almost sure LSD of $\tilde{W}_n$ exists and is equal to $\mu_{zhu}$. Further if $\lim\int M_{2,n}(x_1,x_2) \ dx_1 \to 1$, then the limit is semi-circular. 
 Thus we get Result \ref{result:Zhu} as a special case of Result \ref{res:graphons}.
 
%

%
\begin{remark}
Note that under Assumption A, for every coloured rooted tree $T^{\prime}$,
$$t(T^{\prime}, \{M_{2k,n}\}) \rightarrow 
\int_{[0,1]^{b+1}} \underset{0\leq i<j\leq b} {\prod_{(i,j)\in E}} g_{2|E(i,j)|}(x_{i},x_j) \ \prod_{0 \leq i \leq b} dx_i.$$
As seen in the proofs of Theorem \ref{thm:maingk} and Lemma \ref{lem:wig3}, these trees correspond to the words in $SS(2k)$ and thus the existence of the limits in the above equation implies their positive contribution in the limiting moments.
Note that the uniform convergence of $\{g_{2k,n}\}$ to $\{g_{2k}\}$ and their integrability in Assumption A is a sufficient condition for \eqref{gen-homomorphism density}.
\end{remark}
\subsection{Matrices with a variance profile}\label{variance prf}
\begin{result}\label{res:variance profile}
Consider a Wigner matrix with a variance profile $\sigma$, i.e. suppose the entries of the matrix $W_n$ are $\{y_{ij,n}=\sigma(i/n,j/n)x_{ij,n}; i\leq j\}$ where $\{x_{ij,n};i\leq j\}$ are i.i.d. for every fixed $n$  that satisfy the two conditions (i) and (ii) of Corollary \ref{thm:main} with $t_n=\infty$ and $\sigma$ is a bounded piecewise continuous function on $[0,1]^2$. 
 Then the ESD of $W_n$ converges weakly almost surely to a symmetric probability measure $\nu$ whose $2k$th moment is determined by $\sigma$ and $\{C_{2m}\}_{1\leq m\leq 2k}$. 
\end{result}
\begin{proof}
Indeed, observe that $y_{ij,n}$ satisfy Assumption A with $g_{2k}\equiv \sigma^{2k}C_{2k}$. 
 Thus using Theorem \ref{thm:maingk}, we conclude that the ESD of $W_n$ converges weakly almost surely to a symmetric probability measure $\nu$. 
 
 Now we give a description of the limiting moments. observe that, from Step 3 in the proof of Theorem \ref{thm:maingk}, for each word in $SS_{b}(2k)$ with each distinct letter appearing $s_1,s_2,\ldots,s_b$ times, its contribution to the limiting moments is (see \eqref{limit-ss2k})
$$\int_{[0,1]^{b+1}} \prod_{j=1}^b \sigma^{s_j}(x_{t_j},x_{l_j}) \ \prod_{i\in S} dx_i \prod_{j=1}^b C_{s_j}$$
where $(t_j,l_j)$ denotes the position of first appearance of the $j$th distinct letter in the word.

Hence the $2k$th moment of $\nu$ is given as follows:
\begin{align*}
\beta_{2k}= \displaystyle \sum_{b=1}^k \sum_{\pi \in SS_b(2k)}\int_{[0,1]^{b+1}} \prod_{j=1}^b \sigma^{s_j}(x_{t_j},x_{l_j}) \ \prod_{i\in S} dx_i \prod_{j=1}^b C_{s_j}.
\end{align*}
\end{proof}


\subsection{Band and Random Block matrices: some extensions}\label{band and block}
\subsubsection{Band matrices}\label{band}
In band matrices, the entries are non-zero only around the diagonal in the form of a band. As the dimension of the matrices increase, so does  the number of non-zero elements around the diagonal. Let $m_n$ be as sequence of positive integers such that $m_n \rightarrow \infty$ and $m_n/n \rightarrow \alpha>0$ as $n \rightarrow \infty$. Now we define two modes of banding--\textit{periodic banding} and \textit{non-periodic banding}.\\

\noindent \textbf{Periodic banding}: Periodic banding $W_n^b$ of $W_n$ is the symmetric matrix with entries $y_{ij,n}$ where for $m_n\leq n/2$, 
\begin{align}\label{band1}
y_{ij,n}= \begin{cases}
x_{ij,n} & \ \ \text{ if } |i-j|\leq m_n \ \ \text{ or }|i-j|\geq n- m_n\\
0   & \ \ \text{ otherwise.}
\end{cases}
\end{align}
\noindent \textbf{Non-periodic banding}: Non-periodic banding $W_n^B$ of $W_n$ is the symmetric matrix with entries $y_{ij,n}$ where
\begin{align}\label{band2}
y_{ij,n}= \begin{cases}
x_{ij,n} & \ \ \text{ if } |i-j|\leq m_n\\
0   & \ \ \text{ otherwise.}
\end{cases}
\end{align}
These types of band Wigner matrices have been dealt with in previous works for example in \citep{casati1993wigner} where the authors conjectured the convergence of the ESD of periodic and non-periodic band Wigner matrices to the semi-circular distribution in probability under the Lindeberg condition. Also such matrices were considered in \citep{molchanov1992limiting} where $m_n/n\rightarrow 0$. Later in \citep{casati1993generalized}, the authors considered a more generalized model for Wigner band matrices and showed the convergence of their ESD's in the almost sure sense.


\begin{result}\label{res:band1}
Suppose the random variables $\{x_{ij,n}\}$ satisfy assumption A. Then the ESD of $W_n^b$ (as described in \eqref{band1}) converges weakly almost surely to a symmetric probability measure $\mu_{\alpha}$ whose moments are determined by $\{g_{2k}\}$ and $\alpha= \lim_{n \rightarrow \infty}m_n/n$.
\end{result}
\begin{proof}
For every $n$, define the function $f_n$ on $[0,1]^2$ by 
\begin{align}
f_n(x,y)= \begin{cases}
1 & \ \ \text{ if }|x-y| \leq m_n/n \ \ \text{ or } \ \ |x-y| \geq 1-m_n/n\\
0 & \ \ \text{ otherwise.}
\end{cases}
\end{align}
Observe that the entries $y_{ij,n}$ of $W_n^b$ can be written as $f_{n}(i/n,j/n)x_{ij,n}$.

As $|f_n|\leq 1$, following the Steps 2 and 4 in the proof of Theorem \ref{thm:maingk} the fourth moment condition and Carleman's condition follows immediately. Next observe that $\int_{[0,1]^2}f_n(x,y)\ dxdy$ converges to $\int_{[0,1]^2}f(x,y) \ dxdy$ on $[0,1]^2$ where $f$ is defined on $[0,1]^2$ as follows:
 \begin{align}
 f(x,y)= \begin{cases}
 1 & \ \ \text{ if } |x-y|\leq \alpha \ \ \text{ or } |x-y| \geq 1- \alpha\\
0 & \ \ \text{ otherwise.}
 \end{cases}
 \end{align}
 So we have that $n \mathbb{E}[y_{ij,n}^{2k}]= f_n^{2k}(i/n,j/n)g_{2k,n}(i/n,j/n)$ and $f_ng_{2k,n}$ converges to $fg_{2k}$. 
 Now following the proof of Step 3 in Theorem \ref{thm:maingk}, we get that only $SS(2k)$ words contribute in the limiting moments and for each word in $SS(2k)$ with $b$ distinct letters, its contribution to the limiting moments is as follows:
$$\int_{[0,1]^{b+1}}\prod_{j=1}^b \bigg[g_{s_j}(x_{t_j},x_{l_j})\big[\boldsymbol {1}( |x_{t_j}-x_{l_j}|\geq 1-\alpha)+ \boldsymbol {1}( |x_{t_j}-x_{l_j}|\leq \alpha)\big]\bigg]\ dx_{t_1}dx_{l_1}\cdots dx_{l_b}.
$$ 
Hence the ESD of $W_n^b$ converges weakly almost surely to a symmetric probability measure $\mu_{\alpha}$.
\end{proof}
\begin{remark}
Suppose $\{x_{ij,n}; i\leq j\}$ in Result \ref{res:band1} are i.i.d for every fixed n and satisfies \eqref{gkeven} and \eqref{gkodd}. Then from the previous result we get that the ESD of $W_n^b$ converges weakly almost surely to a symmetric probability measure $\mu_\alpha$ whose $2k$th moment is given as follows:
\begin{align*}
\beta_{2k}(\mu_{\alpha})= \displaystyle \sum_{\pi \in SS(2k)}(2\alpha)^{|\pi|} C_{\pi}.
\end{align*}
Now if the entries of the matrix are $\{y_{ij,n}/\sqrt{n}\}$ where $\{y_{ij,n}\}$ are i.i.d  for all $n$ with finite mean and variance $\sigma^2$, then $C_2= \sigma^2$ and $C_{2k}=0$ for all $k \geq 2$. Therefore we have that the ESD of $W_n^b$ converges to a symmetric probability measure $\mu_\alpha$ whose $2k$th moment is given as follows:
\begin{align*}
\beta_{2k}(\mu_{\alpha})= \displaystyle \sum_{\pi \in SS_k(2k)}(2\alpha)^{|\pi|} \sigma^{2k}.
\end{align*}
Hence $\mu_{\alpha}$ in this case is the semicircular distribution with variance $2\alpha \sigma^2$. Hence we can conclude the convergence in Theorem 4 of \citep{casati1993wigner}, which happens in the almost sure sense.
\end{remark}

\begin{result}\label{res:band2}
Suppose the random variables $\{x_{ij,n}\}$ satisfy assumption A. Then the ESD of $W_n^B$ (as described in \eqref{band2}) converges weakly almost surely to a symmetric probability measure $\mu_{\alpha}$ whose moments are determined by $g_{2k}$ and $\alpha= \lim_{n \rightarrow \infty}m_n/n$.
\end{result}
\begin{proof}
The proof is very similar to the proof of Result \ref{res:band1}. In this case observe that the function $f_n$ on $[0,1]^2$ is given by 
\begin{align}
f_n(x,y)= \begin{cases}
1 & \ \ \text{ if }|x-y| \leq m_n/n\\
0 & \ \ \text{ otherwise}
\end{cases}
\end{align}

 and $\int_{[0,1]^2}f_n(x,y) \ dxdy$ converges to $\int_{[0,1]^2}f(x,y) \ dxdy$ on $[0,1]^2$ where $f$ is defined on $[0,1]^2$ as follows:
 \begin{align}
 f(x,y)= \begin{cases}
 1 & \ \ \text{ if }|x-y| \leq \alpha\\
0 & \ \ \text{ otherwise.}
 \end{cases}
 \end{align}
  Now in a similar manner as the proof of Result \ref{res:band1}, we get that only $SS(2k)$ words contribute in the limiting moments and for each word in $SS(2k)$ with $b$ distinct letters, its contribution to the limiting moments is 
$\int_{[0,1]^{b+1}}\prod_{j=1}^b \big[g_{s_j}(x_{t_j},x_{l_j})\boldsymbol {1}(|x_{t_j}-x_{l_j}|\leq \alpha)\ dx_{t_1}dx_{l_1}\cdots dx_{l_b}.
$ 
 
Therefore the ESD of $W_n^B$ converges weakly almost surely to a symmetric probability measure $\mu_{\alpha}$.
\end{proof}

\subsubsection{Random Block matrices}\label{block}

Random block matrices with finite number of rectangular blocks have been studied in \citep{bolla2004distribution}, \citep{ding2014spectral} and \citep{zhu2020graphon}. In \citep{zhu2020graphon}, the graphon approach has been used to prove the results on  random block matrices. We prove a general result on block matrices by the methods of proof of Theorem \ref{thm:maingk}. Theorem 6.1 in \citep{zhu2020graphon} follows from it. 
\vspace{3mm}

\noindent 
Consider the symmetric matrix $W_n^{\prime}$ consisting of $d^2$ rectangular blocks, $W_n^{\prime (m,l)}, 1\leq m, l \leq d$. We write $W_n^{\prime}= \displaystyle \sum_{m,l} E_{ml} \otimes W_n^{\prime (m,l)}$ where $\otimes$ is the Kronecker product  of matrices and $E_{ml}$ are $d \times d$ elementary matrices with entry 1 at the $(m,l)$th position and $0$ otherwise. \\

\noindent \textbf{Assumption B}: 
\begin{enumerate}
\item[(i)] The blocks $W_n^{\prime (m,l)}, 1\leq m\leq l \leq d$ are $n_m\times n_l$ rectangular random matrices with i.i.d. entries  (inside each block) but independent across blocks, subject to symmetry. 

\item[(ii)]  Also suppose the entries $x_{ij,n}$ have all moments finite. For each $k \in \mathbb{N}$ and $1\leq m\leq l \leq d$,
	 \begin{eqnarray}
 \lim_{n \rightarrow \infty} n \ \mathbb{E}\left[x_{ij,n}^{2k}\right]&=:&C_{2k}^{(m,l)} < \infty  \ \ \text{whenever } x_{ij,n} \ \ \text{is in the} \ (m,l)\text{th block}, \label{ckeven-b}\\
 \lim_{n \rightarrow \infty} n^{\alpha} \  \mathbb{E}\left[x_{ij,n}^{2k-1}\right] &=& 0 \ \ \text{for any}\ \  \alpha < 1.\label{ckodd-b}
\end{eqnarray}	
\end{enumerate}

\begin{result}\label{res:block}
Let $\lim_{n \rightarrow \infty} \frac{n_m}{n} = \alpha_m>0, 1 \leq m \leq d$. 
Suppose $W_n^{\prime}$ is the $n \times n$ symmetric random block matrix such that Assumption B holds. Then the ESD of $W_n^{\prime}$ converges weakly almost surely to a symmetric probability measure $\tilde{\mu}$ whose moments are determined by $(C_{2k}^{(m,l)})_{k\geq 1}$ and $(\alpha_m)_{m=1}^d$.
\end{result}
\begin{proof}
The proof follows in a manner very similar to the proof of Theorem \ref{thm:maingk}. We omit the details and give only the outline of the proof to verify the first moment condition.

Let $C_{2k,n}^{(m,l)}=   n \ \mathbb{E}\left[x_{ij,n}^{2k}\right]  \ \ \text{ whenever } x_{ij,n} \ \ \text{ is in the } (m,l)\text{th block}$. Also let $n_0=0$ and $\alpha_0=0$. Observe that in this case the sequence of functions $g_{2k,n}$ are given as follows : 
$$g_{2k,n}(x,y)= C_{2k,n}^{(m,l)} \ \ \text{ when } (x,y)\in \bigg[\sum_{t=0}^{m-1}n_t/n, \sum_{t=0}^{m}n_t/n \bigg] \times \bigg[\sum_{t=0}^{l-1}n_t/n, \sum_{t=0}^{l}n_t/n\bigg].$$

This converges to $g_{2k}$ which is defined as: 
$$g_{2k}(x,y)= C_{2k}^{(m,l)} \ \ \text{ when } (x,y)\in \bigg[\sum_{t=0}^{m-1}\alpha_t, \sum_{t=0}^{m}\alpha_t \bigg] \times \bigg[\sum_{t=0}^{l-1}\alpha_t, \sum_{t=0}^{l} \alpha_t\bigg].$$

 Now following the same arguments as in the proof of Step 3 in Theorem \ref{thm:maingk} with the above functions $\{g_{2k}\}_{k \geq 1}$, we find  that the limiting moments are determined by $(C_{2k}^{(m,l)})_{k\geq 1}$ and $(\alpha_m)_{m=1}^d$. Hence the first moment condition can be verified.
\end{proof}

Now suppose the entries of $W_n^{\prime}$ are $\frac{x_{ij}}{\sqrt{n}}$ and conditions (1)--(4) of Section 6 in \citep{zhu2020graphon} hold. Then it can be shown easily that only words with $k$ distinct letters contribute to the limiting $2k$th moment. Moreover, the words that do not belong to $SS(2k)$ contribute 0 in the limit. As $SS_k(2k)$ is the collection of all \textit{Catalan words}, only these contribute to the limiting moments. Hence we get Theorem 6.1 of \citep{zhu2020graphon} as a special case of Theorem \ref{thm:maingk}.
\subsection{Simulations}\label{simulations}
Figure \ref{fig:wigfig2} provides the simulated the ESD for some choices of the matrices. 
\begin{figure}[htp]
\begin{subfigure}{.5\textwidth}
  \centering
  \includegraphics[width=.7\linewidth]{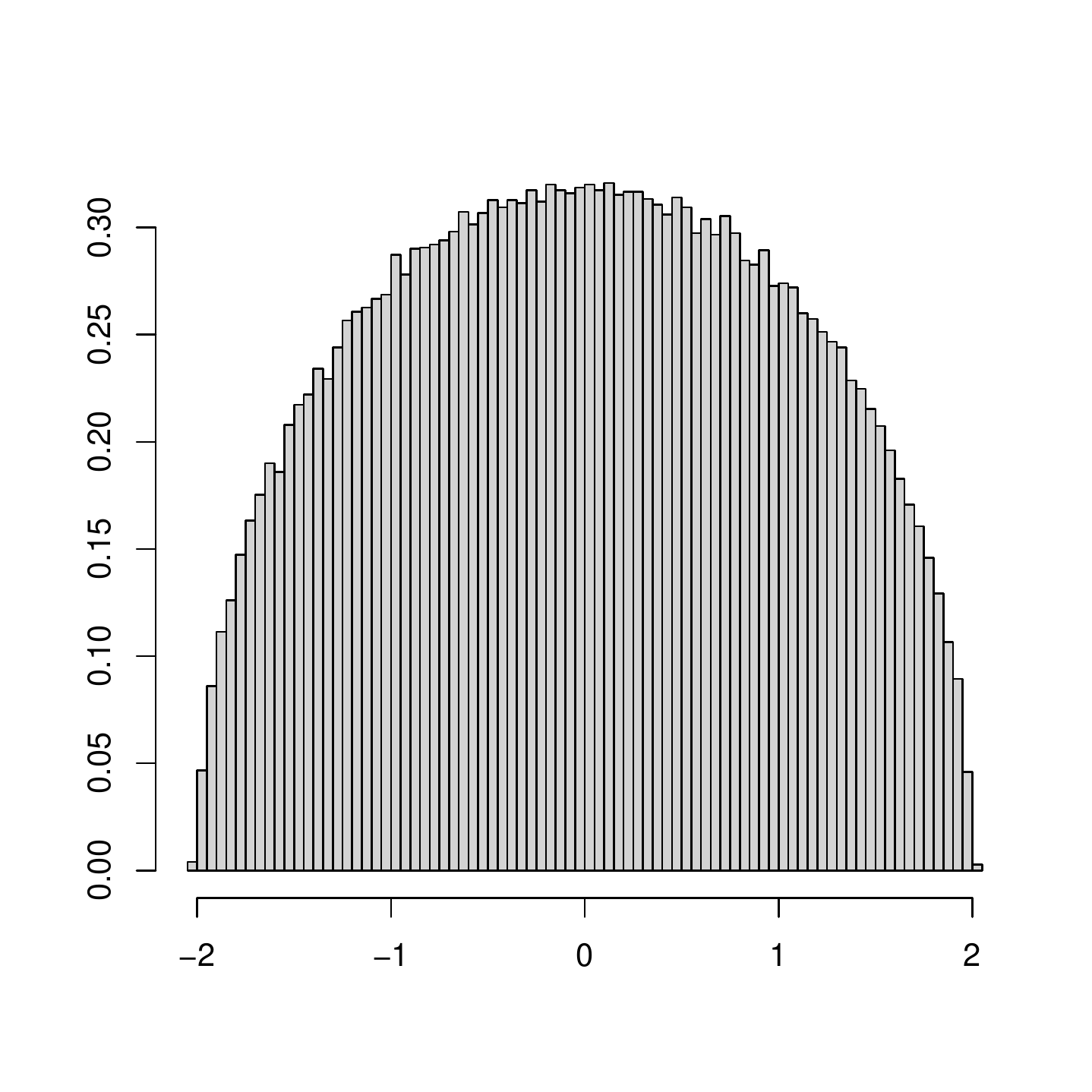}  
  \caption{Input is i.i.d $N(0,1)/\sqrt{n}$.}
\end{subfigure}
\begin{subfigure}{.5\textwidth}
  \centering
  \includegraphics[width=.7\linewidth]{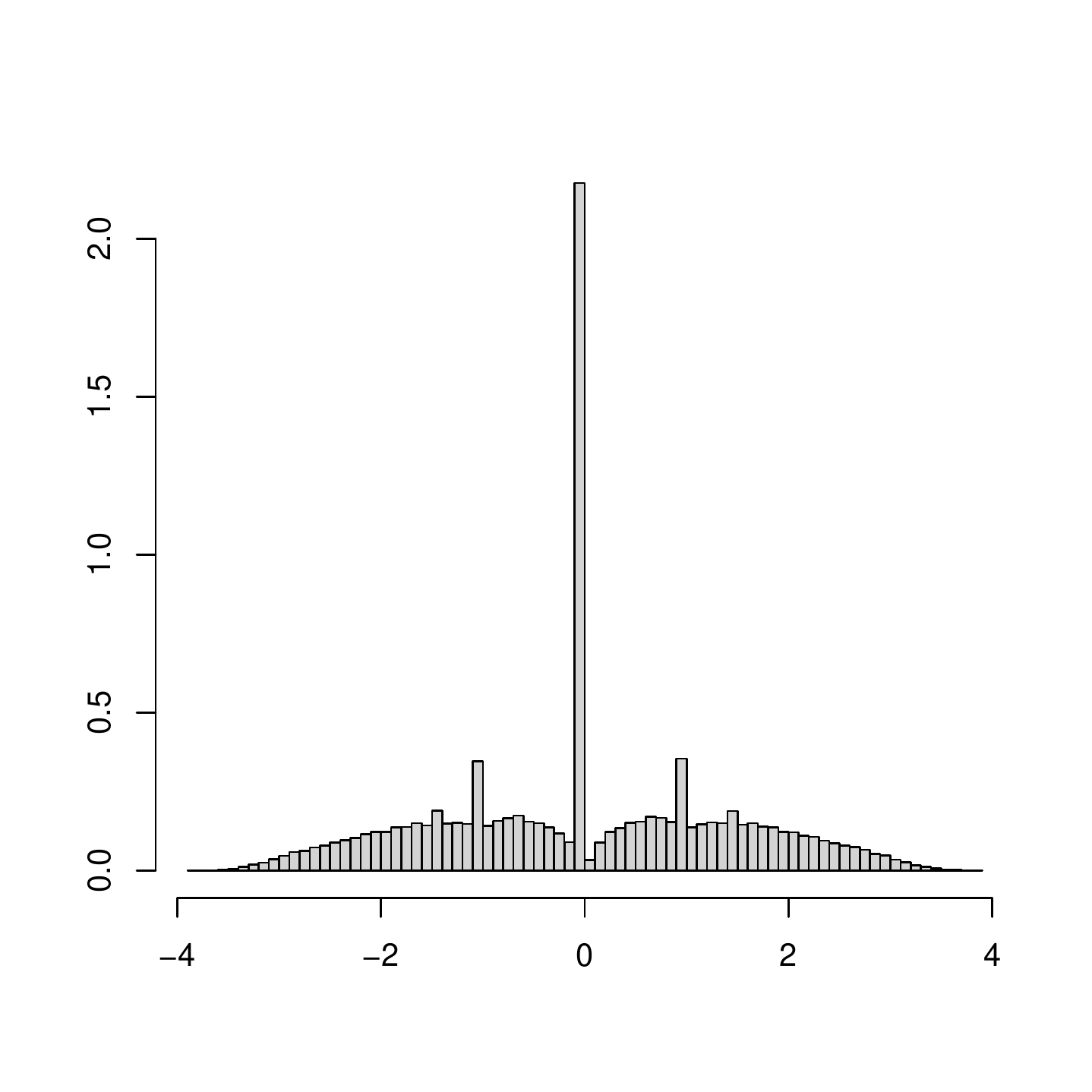}  
  \caption{Input is i.i.d Ber$(2/n)$ for every $n$.}
\end{subfigure}
\\
\begin{subfigure}{.5\textwidth}
  \centering
  \includegraphics[width=.7\linewidth]{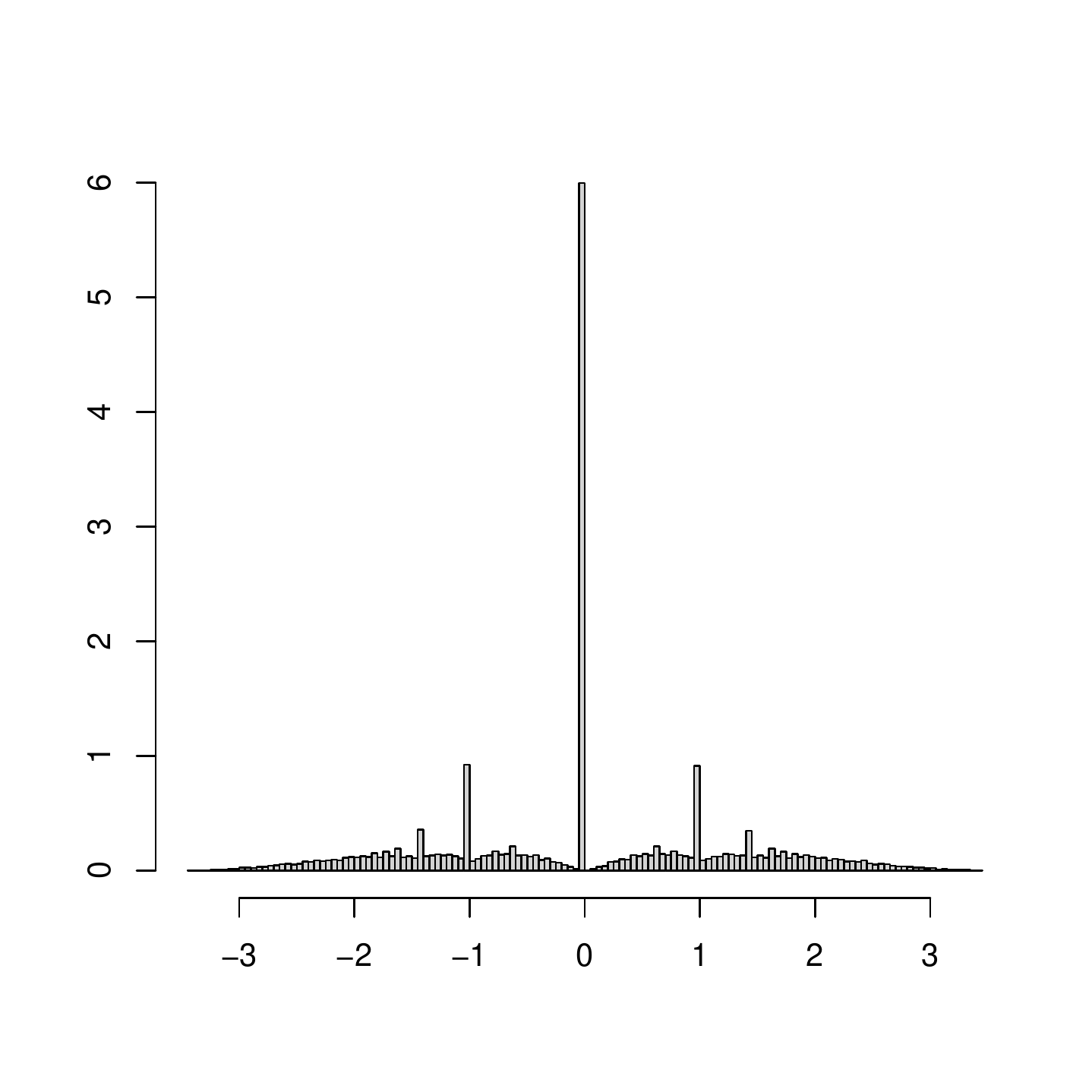}  
  \caption{Input is $x_{ij,n}= \mbox{Ber} (\sin(\frac{\pi(i+j)}{n^2}))$ for every $n$.}
\end{subfigure}
\begin{subfigure}{.5\textwidth}
 \centering
  \includegraphics[width=.7\linewidth]{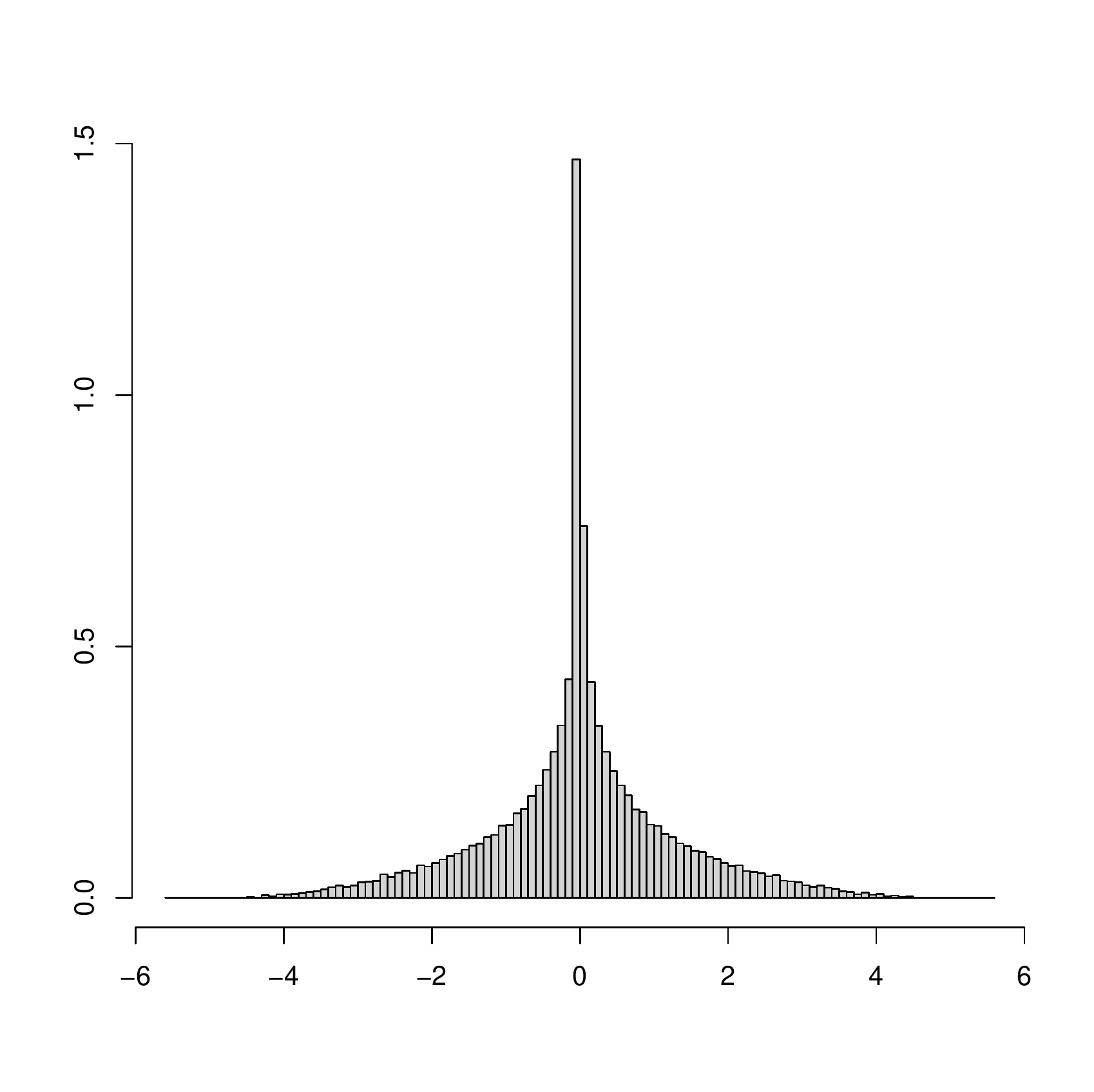}  
  \caption{Input is $x_{ij,n}= \frac{(i+j)^2}{2n^2}\mbox{Ber}(3/n)$ for every $n$.}
\end{subfigure}
\caption{Histogram of the eigenvalues of $W_n$ for $n=1000, 30$ replications.}
\label{fig:wigfig2}
\end{figure}
\vspace{5mm}


\providecommand{\bysame}{\leavevmode\hbox to3em{\hrulefill}\thinspace}
\providecommand{\MR}{\relax\ifhmode\unskip\space\fi MR }
\providecommand{\MRhref}[2]{%
  \href{http://www.ams.org/mathscinet-getitem?mr=#1}{#2}
}
\providecommand{\href}[2]{#2}

%

\end{document}